\documentclass{amsart}

\usepackage[pdftex]{graphicx} 

\usepackage{amssymb,amsxtra,latexsym}

\usepackage[osf]{newpxtext}
\usepackage{textcomp}
\usepackage[cmintegrals]{newpxmath}
\usepackage[scr=rsfs]{mathalfa}

\usepackage[varqu,varl]{inconsolata} 

\usepackage[cmtip,matrix,arrow,curve]{xy}
\CompileMatrices
\newdir{((}{{\lhook\kern-.1em}}
\newdir{))}{{\rhook\kern-.1em}}
\newdir{(o}{{\lhook\kern-.1em}}
\newdir{o)}{{\rhook\kern-.1em}}
\newdir{ >}{{}*!/-5pt/@{>}}

\usepackage{booktabs}

\usepackage{manfnt}

\makeatletter

\def\theglossary{\@restonecoltrue\if@twocolumn\@restonecolfalse\fi
\columnseprule\z@ \columnsep 35\p@
\let\@makeschapterhead\indexchap
\@xp\chapter\@xp*\@xp{\glossaryname}%
\thispagestyle{plain}%
\let\item\@idxitem
\parindent\z@  \parskip\z@\@plus.3\p@\relax
\footnotesize}
\def\glossaryname{Notation Index}
\makeatother

\makeglossary

\usepackage{hyperref,xcolor}

\definecolor{ceruleanblue}{rgb}{0.16, 0.32, 0.75}
\hypersetup{colorlinks=true,allcolors=ceruleanblue}

\renewcommand\theenumi{\roman{enumi}}

\newcommand\note[1]%
{$^\dagger$\marginpar{\footnotesize{$^\dagger${#1}}}}

\divide\count253 by 60 %hours after midnight, rounded down
\divide\count255 by 60
\multiply\count255 by 60 %minutes after midnight, rounded down
\advance\count254 by -\count255 %minutes after the hour

\def\today{\number\year-\ifnum\month<10
0\fi\number\month-\ifnum\day<10 0\fi\number\day}

\def\hour{\ifnum\count253<10
0\number\count253\else\number\count253\fi}

\def\minute{\ifnum\count254<10
0\number\count254\else\number\count254\fi}

\makeatletter
\newcounter{alinea}
\numberwithin{alinea}{section}

\newenvironment{alinea}[1][]{\par\smallskip\noindent%
\refstepcounter{alinea}\thealinea.~\ifx\newenvironment#1\newenvironment%
\else\textbf{#1.}~\fi}{\smallskip}
\makeatother

\makeatletter
\newenvironment{warning}[1][]{
\begin{trivlist}\item\noindent%
      \begingroup\hangindent=2pc\hangafter=-2%
      \clubpenalty=10000%
      \hbox to0pt{\hskip-\hangindent\manfntsymbol{127}\hfill}\ignorespaces%
      \refstepcounter{alinea}\thealinea.%
      \let\p@@r=\par\def\p@r{\p@@r\hangindent=0pc}\let\par=\p@r}%
    {\hspace*{\fill}%$\lrcorner$
\endgraf\endgroup\end{trivlist}}
\makeatother

\numberwithin{equation}{alinea}

\makeatletter
\let\oldtagform@\tagform@
\renewcommand\tagform@[1]{\maketag@@@{\ignorespaces#1\unskip\@@italiccorr.}}
\renewcommand{\eqref}[1]{\textup{\oldtagform@{\ref{#1}}}}
\makeatother

\swapnumbers

\newtheorem{theorem}[alinea]{Theorem}
\newtheorem{lemma}[alinea]{Lemma}

\newtheorem{corollary}[alinea]{Corollary}

\newtheorem{subtheorem}[equation]{Theorem}
\newtheorem{sublemma}[equation]{Lemma}
\newtheorem{subproposition}[equation]{Proposition}
\newtheorem{subcorollary}[equation]{Corollary}

\theoremstyle{definition}

\newtheorem{subexample}[equation]{Example}

\newcommand\lie{\mathfrak}

\renewcommand{\a}{\lie{a}}
\newcommand{\g}{\lie{g}}
\newcommand{\h}{\lie{h}}
\newcommand{\m}{\lie{m}}
\newcommand{\n}{\lie{n}}
\newcommand{\p}{\lie{p}}
\newcommand{\q}{\lie{q}}
\renewcommand{\t}{\lie{t}}

\newcommand\bb[1]{{\text{\bf#1}}}

\newcommand\Z{\bb{Z}} 
\newcommand\Q{\bb{Q}}
\newcommand\R{\bb{R}} 
\newcommand\C{\bb{C}}

\newcommand\T{\bb{T}}

\newcommand\ca{\mathscr}

 \DeclareFontFamily{OT1}{pzc}{}
 \DeclareFontShape{OT1}{pzc}{m}{it}{<-> s * [1.100] pzcmi7t}{}
 \DeclareMathAlphabet{\mathpzc}{OT1}{pzc}{m}{it}
\newcommand\func[1]{\operatorname{\mathrm{#1}}}

\newcommand\Ad{\func{Ad}}

\newcommand\coker{\func{coker}}

\newcommand\id{\func{id}}
\newcommand\im{\func{im}}

\newcommand\pr{\func{pr}}

\newcommand\Lie{\func{Lie}}

\newcommand\group[1]{{\text{\bf#1}}}

\newcommand\U{\group{U}}

\newcommand\abs[1]{\lvert#1\rvert}

\newcommand\inner[1]{\langle#1\rangle}

\newcommand\qu[1][\kern.3ex]{/\kern-.7ex/_{\kern-.4ex#1}}
\newcommand\bigqu[1][\,\,]{\big/\kern-.85ex\big/_{\!\!#1}}

\newcommand\powl{[\kern-.3ex[}
\newcommand\powr{]\kern-.3ex]}
\newcommand\bigpowl{\bigl[\kern-.6ex\bigl[}
\newcommand\bigpowr{\bigr]\kern-.6ex\bigr]}

\newcommand\sur{\mathrel{\to\kern-1.8ex\to}}
\newcommand\iso{\mathrel{\hookrightarrow\kern-1.8ex\to}}

\newcommand\longto{\longrightarrow}

\newcommand\longsur{\mathrel{\longrightarrow\kern-1.8ex\to}}
\newcommand\Edh{{D\mkern-15mu\text{\raise0.1em\hbox{--}}\mkern8mu}}

\newcommand\zerodots%
{\smash{\makebox[0pt]{$_{\lower8pt\hbox{\Large$0$}}%
\ddots^{\lower3pt\hbox{\Large$0$}}$}}}
\newcommand\bigzerodots%
{\smash{\makebox[0pt]{$_{\lower6pt\hbox{\Large$0$}}%
\ddots^{\hbox{\Large$0$}}$}}}

\renewcommand\subset{\subseteq}
\renewcommand\supset{\supseteq}

\newcommand\bas{{\mathrm{bas}}}

\begin{document}

%%%%%%%%%%%%%%%%%%%%%%%%%%%%%%%%%%%%%%%%%%%%%%%%%%%%%%%%%%%%%%%%%%%%%%%%
%%%%%%%%%%%%%%%%%%%%%%%%%%%%%%%%%%%%%%%%%%%%%%%%%%%%%%%%%%%%%%%%%%%%%%%%

\title{Convexity properties of presymplectic moment maps}

\author{Yi Lin}

\email{yilin@georgiasouthern.edu}

\address{Department of Mathematical Sciences, Georgia Southern
  University, Statesboro, GA 30460 USA}

\author{Reyer Sjamaar}

\email{sjamaar@math.cornell.edu}

\address{Department of Mathematics, Cornell University, Ithaca, NY
14853-4201, USA}

%\subjclass[2010]{}

\date\today

%%%%%%%%%%%%%%%%%%%%%%%%%%%%%%%%%%%%%%%%%%%%%%%%%%%%%%%%%%%%%%%%%%%%%%%%
%%%%%%%%%%%%%%%%%%%%%%%%%%%%%%%%%%%%%%%%%%%%%%%%%%%%%%%%%%%%%%%%%%%%%%%%

\begin{abstract}
The convexity and Morse-theoretic properties of moment maps in
symplectic geometry typically fail for presymplectic manifolds.  We
find a condition on presymplectic moment maps that prevents these
failures.  Our result applies for instance to Prato's quasifolds and
to Hamiltonian actions on contact manifolds and cosymplectic
manifolds.
\end{abstract}

%%%%%%%%%%%%%%%%%%%%%%%%%%%%%%%%%%%%%%%%%%%%%%%%%%%%%%%%%%%%%%%%%%%%%%%%
%%%%%%%%%%%%%%%%%%%%%%%%%%%%%%%%%%%%%%%%%%%%%%%%%%%%%%%%%%%%%%%%%%%%%%%%

\maketitle

\tableofcontents

%%%%%%%%%%%%%%%%%%%%%%%%%%%%%%%%%%%%%%%%%%%%%%%%%%%%%%%%%%%%%%%%%%%%%%%%
%%%%%%%%%%%%%%%%%%%%%%%%%%%%%%%%%%%%%%%%%%%%%%%%%%%%%%%%%%%%%%%%%%%%%%%%

%%%%%%%%%%%%%%%%%%%%%%%%%%%%%%%%%%%%%%%%%%%%%%%%%%%%%%%%%%%%%%%%%%%%%%%%
\section{Introduction}
%%%%%%%%%%%%%%%%%%%%%%%%%%%%%%%%%%%%%%%%%%%%%%%%%%%%%%%%%%%%%%%%%%%%%%%%

\begin{alinea}
This paper deals with a topic in transverse geometry: in the context
of a manifold $X$ with a (regular) foliation $\ca{F}$ and a symplectic
structure transverse to the foliation we develop analogues of a few
basic results of symplectic geometry.  While statements such as the
Darboux theorem remain valid, one quickly discovers counterexamples to
naive parallels of the convexity theorems of Hamiltonian compact Lie
group actions proved by Atiyah \cite{atiyah;convexity-commuting},
Guillemin and Sternberg \cite{guillemin-sternberg;convexity;;1982},
and Kirwan \cite{kirwan;convexity-III}.  Some such counterexamples are
catalogued in Section \ref{section;example}.  Our main contribution is
to state a condition under which these convexity theorems are true in
the transversely symplectic setting.  The condition, which we call
\emph{cleanness} of the group action, and special cases of which have
been found earlier by other authors, such as He
\cite{he;odd-dimensional} and Ishida \cite{ishida;transverse-kaehler},
is that there should exist an ideal of the Lie algebra of the group,
called the \emph{null ideal}, which at every point of the manifold
spans the tangent space of the intersection of the group orbit with
the leaf of $\ca{F}$.

We state and prove our convexity theorem in Section
\ref{section;convex}.  It is formulated in terms of presymplectic
structures instead of the equivalent language of transversely
symplectic foliations.  The adjective ``presymplectic'' has
conflicting meanings in the current literature.  We will use it for a
\emph{closed} $2$-form of \emph{constant rank}.
\end{alinea}

\begin{alinea}
It was proved by Atiyah \cite{atiyah;convexity-commuting} and
Guillemin and Sternberg \cite{guillemin-sternberg;convexity;;1982}
that the components of a symplectic moment map are Morse-Bott
functions, an observation that lies at the heart of all subsequent
developments in equivariant symplectic geometry.  In the hope of
opening the way to similar applications to the topology of
presymplectic Lie group actions we show in Section \ref{section;morse}
that under the cleanness assumption the components of a presymplectic
moment map are Morse-Bott as well.  In Section \ref{section;example}
we discuss some examples of the convexity theorem, including
orbifolds, contact manifolds and Prato's quasifolds
\cite{prato;non-rational-symplectic}.
\end{alinea}

\begin{alinea}
To what extent does the moment polytope of the action of a compact Lie
group $G$ on a transversely symplectic manifold $X$ depend only on the
leaf space $X/\ca{F}$?  If the foliation $\ca{F}$ is fibrating, the
leaf space is a symplectic manifold and the moment polytope of $X$ is
the same as that of $X/\ca{F}$.  In this case the moment polytope of
$X$ is therefore completely determined by $X/\ca{F}$.  But if the null
foliation is not fibrating, the leaf space is often a messy
topological space from which one cannot hope to recover the polytope.
The structure of the leaf space can be enriched to that of an \'etale
symplectic stack $\ca{X}$ (as in Lerman and Malkin
\cite{lerman-malkin;deligne-mumford}, but without the Hausdorff
condition), which is equipped with a Hamiltonian action of a stacky
Lie group $\ca{G}$ (namely the ``quotient'' of $G$ by the in general
non-closed normal subgroup generated by the null ideal).  We will show
in a future paper \cite{hoffman-lin-sjamaar;hamiltonian-stack} that
the moment polytope, reinterpreted as a stacky polytope, is an
intrinsic invariant of the $\ca{G}$-action on $\ca{X}$.
\end{alinea}

\begin{alinea}
The authors are grateful to Zuoqin Wang and Jiang-Hua Lu for helpful
discussions.  We draw the reader's attention to Ratiu and Zung's paper
\cite{ratiu-zung;presymplectic}, which overlaps with ours and appeared
on the arXiv at nearly the same time.  We have omitted from this
revised version some material developed more fully by them.
\end{alinea}  

%%%%%%%%%%%%%%%%%%%%%%%%%%%%%%%%%%%%%%%%%%%%%%%%%%%%%%%%%%%%%%%%%%%%%%%%
\section{Presymplectic convexity}\label{section;convex}
%%%%%%%%%%%%%%%%%%%%%%%%%%%%%%%%%%%%%%%%%%%%%%%%%%%%%%%%%%%%%%%%%%%%%%%%

\begin{alinea}\label{alinea;presymplectic}
A \emph{presymplectic manifold} is a paracompact $C^\infty$-manifold
equipped with a closed $2$-form of constant rank.  A \emph{Hamiltonian
  action} on a presymplectic manifold $(X,\omega)$ consists of two
pieces of data: a smooth action of a Lie group $G$ on $X$ and a smooth
\emph{moment map} $\Phi\colon X\to\g^*$.  Here $\g=\Lie(G)$ denotes
the Lie algebra of $G$ and $\g^*$ the dual vector space of $\g$.
These data are subject to the following requirements: the $G$-action
should preserve the presymplectic form (i.e.\ $g^*\omega=\omega$ for
all $g\in G$) and $\Phi$ should be an equivariant map satisfying
$d\Phi^{\xi}=\iota(\xi_X)\omega$ for all $\xi\in\g$.  Here $\xi_X$
denotes the vector field on $X$ induced by $\xi\in\g$, and
$\Phi^{\xi}$ is the function on $X$ defined by
$\Phi^\xi(x)=\inner{\Phi(x),\xi}$ for $x\in X$.

In the remainder of this section $X$ will denote a fixed manifold with
a presymplectic form $\omega$ and $G$ will denote a fixed
\emph{compact connected} Lie group acting on $X$ in a Hamiltonian
fashion with moment map $\Phi$.  We will refer to $X$ as a
\emph{presymplectic Hamiltonian $G$-manifold}.  As far as we know,
this notion was first introduced (under a different name) by Souriau
\cite[\S\,11]{souriau;structure-dynamical}.  The main goal of this
section is to establish the following theorem.  This result is very
similar to the symplectic case, but the presence of the null foliation
causes some interesting new phenomena.
\end{alinea}

\begin{theorem}[presymplectic convexity theorem]\label{theorem;convex}
Assume that the $G$-action on $X$ is clean.  Assume also that the
manifold $X$ is connected and that the moment map $\Phi\colon
X\to\g^*$ is proper.  Choose a maximal torus $T$ of $G$ and a closed
Weyl chamber $C$ in $\t^*$, where $\t=\Lie(T)$, and define
$\Delta(X)=\Phi(X)\cap C$.
\begin{enumerate}
\item\label{item;open}
The fibres of $\Phi$ are connected and $\Phi\colon X\to\Phi(X)$ is an
open map.
\item\label{item;convex}
$\Delta(X)$ is a closed convex polyhedral set.
\item\label{item;rational}
$\Delta(X)$ is rational if and only if the null subgroup $N(X)$ of $G$
  is closed.
\end{enumerate}
\end{theorem}

This statement contains several undefined terms, which we proceed to
explain in \ref{alinea;convex}--\ref{alinea;clean} below.  In
\ref{alinea;atiyah}--\ref{alinea;coisotropic} we make further
preliminary comments and state a sequence of auxiliary results.  The
proof of the theorem is in \ref{alinea;proof}.  Examples are presented
in Section \ref{section;example}.

\begin{alinea}\label{alinea;convex}
A \emph{convex polyhedral set} in a finite-dimensional real vector
space is an intersection of a locally finite number of closed
half-spaces.  A \emph{convex polyhedron} is an intersection of a
finite number of closed half-spaces.  A \emph{convex polytope} is a
bounded convex polyhedron.  If the manifold $X$ is compact, the set
$\Delta(X)$ defined in Theorem \ref{theorem;convex} is a convex
polytope.
\end{alinea}

\begin{warning}\label{warning;rational}
Let $E$ be a finite-dimensional real vector space equipped with a
$\Q$-structure.  We call a convex polyhedral subset of $E$
\emph{rational} if it can be written as the locally finite
intersection of half-spaces, each of which is given by an inequality
of the form $\inner{\eta,\cdot}\ge a$ with rational normal vector
$\eta\in E^*(\Q)$ and $a\in\R$.  This is nonstandard terminology.  The
more usual definition requires the scalars $a$ to be rational as well.
Our notion of rationality is equivalent to the normal fan of the
polyhedral set being rational.

We call a convex polyhedral subset of $\t^*$ rational if it is
rational with respect to the $\Q$-structure
$\Q\otimes_\Z\lie{X}^*(T)$.  Here $\lie{X}^*(T)\subset\t^*$ denotes
the character lattice of the torus $T$, i.e.\ the lattice dual to the
exponential lattice $\lie{X}_*(T)=\ker(\exp\colon\t\to T)$.
\end{warning}  

\begin{alinea}[The null ideal sheaf]\label{alinea;null}
The subbundle $\ker(\omega)$ of the tangent bundle $TX$ is involutive
(see e.g.\ \cite[\S\,3]{block-getzler;quantization-foliations}) and
therefore, by Frobenius' theorem, integrates to a (regular) foliation
$\ca{F}=\ca{F}_X$, called the \emph{null foliation} of $\omega$.  We
call $\ker(\omega)$ the \emph{tangent bundle} of the foliation and
denote it usually by $T\ca{F}$.  The leaves of $\ca{F}$ are not
necessarily closed; indeed the case where they are \emph{not} closed
is the focus of our attention.

Let $U$ be an open subset of $X$.  Define $\n(U)$ to be the Lie
subalgebra of $\g$ consisting of all $\xi\in\g$ with the property that
the $1$-form $d\Phi^\xi=\iota(\xi_X)\omega$ vanishes on the
$G$-invariant open set $G\cdot U$.  Equivalently, $\xi$ is in $\n(U)$
if and only if the moment map component $\Phi^\xi$ is locally constant
on $G\cdot U$, which is the case if and only if the induced vector
field $\xi_X$ is tangent everywhere on $G\cdot U$ to the foliation
$\ca{F}$.  Define $N(U)$ to be the connected immersed (but not
necessarily closed) Lie subgroup of $G$ whose Lie algebra is $\n(U)$.
If the leaves of $\ca{F}|_{G\cdot U}$ are closed subsets of $G\cdot
U$, then the subgroup $N(U)$ is closed, but we do not assume this to
be the case.  For all $g\in G$ and $\xi\in\n(U)$ we have
$$
\iota((\Ad_g(\xi))_X)\omega=\iota(g_*(\xi_X))\omega=
(g^{-1})^*(\iota(\xi_X)\omega)=0
$$
on $G\cdot U$, and therefore the adjoint action of $G$ on $\g$
preserves $\n(U)$.  In particular the subalgebra $\n(U)$ of $\g$ is an
ideal and the subgroup $N(U)$ of $G$ is normal.

The assignment $\n\colon U\mapsto\n(U)$ is a presheaf on $X$.  Its
associated sheaf $\tilde{\n}$ is a subsheaf of ideals of the constant
sheaf $\g$.  We have $\tilde{\n}(U)=\prod_V\n(V)$, where the product
is over all connected components $V$ of $U$.  The restriction
morphisms of the presheaf $\n$ are injective, which implies that for a
decreasing sequence of open sets $U_1\supset U_2\supset\cdots\supset
U_n\supset\cdots$ the sequence of ideals
$\n(U_1)\subset\n(U_2)\subset\dots\subset\n(U_n) \subset\dots$ is
increasing and therefore eventually constant.  Hence the sheaf
$\tilde{\n}$ is constructible and its stalk $\n_x=\tilde{\n}_x$ at $x$
is equal to $\n(U)$ for all sufficiently small open neighbourhoods $U$
of $x$.  Similarly, the presheaf $N\colon U\mapsto N(U)$ sheafifies to
a constructible subsheaf $\tilde{N}$ of normal subgroups of the
constant sheaf $G$, whose stalk $N_x=\tilde{N}_x$ at $x$ is equal to
$N(U)$ for any suitably small open $U$ containing $x$.  We will call
$\tilde{\n}=\tilde{\n}_X$ the \emph{null ideal sheaf} and
$\tilde{N}=\tilde{N}_X$ the \emph{null subgroup sheaf} of the
presymplectic Hamiltonian action.
\end{alinea}

\begin{alinea}[Clean actions]\label{alinea;clean}
Let $x\in X$ and let $U$ be a $G$-invariant open neighbourhood of $x$.
The $N(U)$-action maps each leaf of the foliation $\ca{F}|_U$ to
itself.  Therefore the orbit $N(U)\cdot x$ is contained in $G\cdot
x\cap\ca{F}(x)$, where $\ca{F}(x)$ denotes the leaf of $x$.
Infinitesimally, the tangent space $T_x(N(U)\cdot x)$ is contained in
$T_x(G\cdot x)\cap T_x\ca{F}$.  Taking $U$ to be small enough we have
$N(U)=N_x$ and so $T_x(N_x\cdot x)\subset T_x(G\cdot x)\cap
T_x\ca{F}$.  We call the $G$-action on $X$ \emph{clean} at $x$ if this
inclusion is an equality, i.e.\
\begin{equation}\label{equation;clean}
T_x(N_x\cdot x)=T_x(G\cdot x)\cap T_x\ca{F}.
\end{equation}
Cleanness is a $G$-invariant condition: if the action is clean at $x$,
then it is clean at $gx$ for every $g\in G$.  Cleanness is a local
condition: the $G$-action on $X$ is clean at $x$ if and only if the
$G$-action on $U$ is clean at $x$ for some $G$-invariant open set $U$
containing $x$.  Cleanness is not necessarily an open condition.  (See
Example \ref{example;clean}.)

We will state some criteria for the action to be clean in terms of the
induced $G$-action on the leaf space.  Since the $G$-action on $X$
preserves the form $\omega$, it sends leaves to leaves, and therefore
descends to a continuous action on the leaf space $X/\ca{F}$ (equipped
with its quotient topology).  Let $G_x$ be the stabilizer of $x\in X$
and $\g_x$ its Lie algebra.  Let $\bar{x}=\ca{F}(x)$ denote the leaf
of $x$ considered as a point in the leaf space.  Let $G_{\bar{x}}$ be
the stabilizer of $\bar{x}$ with respect to the induced $G$-action on
$X/\ca{F}$.  We equip $G_{\bar{x}}$ with its induced Lie group
structure (see e.g.\ \cite[\S~III.4.5]{bourbaki;groupes-algebres}),
which makes it an immersed (but not necessarily embedded) subgroup of
$G$.  Its Lie algebra $\g_{\bar{x}}$ consists of all $\xi\in\g$
satisfying $\xi_X(x)\in T_x\ca{F}$.  By definition we have
\begin{equation}\label{equation;stabilizer}
G_{\bar{x}}\cdot x=G\cdot x\cap\ca{F}(x)
\end{equation}
and
\begin{equation}\label{equation;null}
\n(U)=\bigcap_{x\in G\cdot U}\g_{\bar{x}}
\end{equation}
for all open $U\subset X$.  (Thus $N(X)$ is the identity component of
the subgroup of $G$ that acts trivially on the leaf space $X/\ca{F}$.)
The foliation $\ca{F}$, being $G$-invariant, induces a foliation of
each orbit $G\cdot x$.  This induced foliation is equal to the null
foliation of the form $\omega$ restricted to $G\cdot x$.  The leaves
of the induced foliation are the connected components of intersections
of the form $G\cdot x\cap\ca{F}(y)$.  We see from
\eqref{equation;stabilizer} that the leaves of the induced foliation
can also be described as the left translates of the connected
components of $G_{\bar{x}}\cdot x$.

\begin{sublemma}\label{lemma;clean-x}
For every $x\in X$ the following conditions are equivalent.
\begin{enumerate}
\item\label{item;clean}
The $G$-action is clean at $x$;
\item\label{item;clean-vector}
there exist vectors $\xi_1$, $\xi_2,\dots$, $\xi_k\in\g$ and a
$G$-invariant open neighbourhood $U$ of $x$ with the property that
$\xi_{1,X}$, $\xi_{2,X},\dots$, $\xi_{k,X}$ are tangent to $\ca{F}$ on
$U$ and $\xi_{1,X}(x)$, $\xi_{2,X}(x),\dots$, $\xi_{k,X}(x)$ span
$T_x(G\cdot x)\cap T_x\ca{F}$;
\item\label{item;foliation}
the leaves of the foliation of the $G$-orbit $G\cdot x$ induced by
$\ca{F}$ are $N_x$-orbits;
\item\label{item;stabilizer-orbit}
$T_x(N_x\cdot x)=T_x(G_{\bar{x}}\cdot x)$;
\item\label{item;leaf}
the orbit $N_x\cdot x$ is an open subset of the orbit
$G_{\bar{x}}\cdot x$;
\item\label{item;stalk}
$\g_{\bar{x}}=\g_x+\n_x$;
\item\label{item;sub-clean}
the $G_{\bar{x}}$-action is clean at $x$.
\end{enumerate}
\end{sublemma}

\begin{proof}
Condition \eqref{item;clean-vector} is a straightforward reformulation
of the definition of cleanness.  Next we show that
\eqref{item;clean}~$\iff$~\eqref{item;foliation}.  Let $\ca{F}^x$ be
the induced foliation of the orbit $G\cdot x$.  Let $g\in G$ and
$y=gx$.  The leaf $\ca{F}^x(y)=g\cdot\ca{F}^x(x)$ contains the orbit
$N_x\cdot y=N_x\cdot gx=g\cdot N_x\cdot x$.  Since $N_x$ is connected,
the reverse inclusion $\ca{F}^x(y)\subset N_x\cdot y$ holds if and
only if $N_x\cdot x$ is open in $\ca{F}^x(x)$, which is the case if
and only if $T_x(N_x\cdot x)=T_x(G\cdot x)\cap T_x\ca{F}$.  The
equivalence \eqref{item;clean}~$\iff$~\eqref{item;stabilizer-orbit} is
immediate from \eqref{equation;clean} and \eqref{equation;stabilizer}.
Since the orbit $N_x\cdot x$ is always a subset of $G_{\bar{x}}\cdot
x$, the equivalence
\eqref{item;stabilizer-orbit}~$\iff$~\eqref{item;leaf} is immediate.
Condition \eqref{item;stabilizer-orbit} is equivalent to
$\g_{\bar{x}}/\g_x=(\n_x+\g_x)/\g_x$, which is equivalent to
\eqref{item;stalk}.  Finally, \eqref{item;sub-clean} is a
reformulation of \eqref{item;stabilizer-orbit}.
\end{proof}

\begin{subexample}\label{example;clean}
Let $X=Y\times V$ and $\omega=\omega_Y\oplus0$, where $(Y,\omega_Y)$
is a symplectic Hamiltonian $G$-manifold and $V$ is a real $G$-module.
Then the leaves of the null foliation $\ca{F}$ are the fibres of the
projection $X\to Y$.  Let $x=(y,v)\in X$.  Then $\ca{F}(x)=\{y\}\times
V$ and therefore for $\xi\in\g$ the tangent vector to the orbit
$\xi_{X,x}\in T_x(G\cdot x)$ is tangent to the leaf if and only if
$\xi_{Y,y}=0$.  This shows that
\begin{equation}\label{equation;gf}
T_x(G\cdot x)\cap T_x\ca{F}=\{0\}\times T_v(G_y\cdot v)\subset
T_yY\times V.
\end{equation}
The stabilizer of the leaf $\ca{F}(x)$ is $G_{\bar{x}}=G_y$.  It
follows from \eqref{equation;null} that for all open $U\subset X$ we
have 
$$
\n(U)=\bigcap_{(y,v)\in G\cdot U}\g_{\bar{x}}=\bigcap_{(y,v)\in G\cdot
  U}\g_y=\lie{k},
$$
where $\lie{k}=\ker(\g\to\Gamma(TY))$ denotes the kernel of the
infinitesimal action of $\g$ on $Y$.  Hence $\tilde{\n}_X$ is the
constant sheaf with stalk $\lie{k}$ and $N_x=K$, where
$K=\exp(\lie{k})$.  Therefore
\begin{equation}\label{equation;n}
T_x(N_x\cdot x)=\{0\}\times T_v(K\cdot v)\subset T_yY\times V.
\end{equation}
Comparing \eqref{equation;gf} with \eqref{equation;n} we see that the
action is clean at $x=(y,v)$ if $v=0$, but usually not at other
points.  For instance, if the action on $Y$ is effective ($K=\{1\}$),
the action is not clean at $x$ soon as $\dim(G_y\cdot v)\ge1$.
\end{subexample}  

Two extreme cases of the cleanness condition merit attention.  We say
the $G$-action is \emph{leafwise transitive at $x$} if
$\ca{F}(x)=N_x\cdot x$.  Leafwise transitivity at a point is a
$G$-invariant, local and open condition.  We say that the action is
\emph{leafwise nontangent at $x$} if $T_x(G\cdot x)\cap T_x\ca{F}=0$.
Leafwise nontangency forces the stalk $N_x$ to be a subgroup of the
stabilizer $G_x$.  Leafwise nontangency at a point is a $G$-invariant
local condition, but not necessarily an open condition.  Either
condition implies cleanness.

We say that the $G$-action on $X$ is \emph{clean}
(resp.\ \emph{leafwise transitive}, resp.\ \emph{leafwise nontangent})
if it is clean (resp.\ leafwise transitive, resp.\ leafwise
nontangent) at all points of $X$.  Leafwise transitivity guarantees
that the null foliation is Riemannian, a property that is frequently
useful in applications.
\end{alinea}

\begin{alinea}\label{alinea;atiyah}
In the symplectic case ($\ker(\omega)=0$) the foliation $\ca{F}$ is
trivial, the sheaf $\tilde{\n}$ is constant with stalk equal to the
kernel of the infinitesimal action $\g\to\Gamma(TX)$, and the
cleanness condition is automatically fulfilled, so Theorem
\ref{theorem;convex} reduces to the convexity theorems of Atiyah
\cite{atiyah;convexity-commuting}, Guillemin and Sternberg
\cite{guillemin-sternberg;convexity;;1982}, and Kirwan
\cite{kirwan;convexity-III}.  A novel feature of our more general
theorem is that the polyhedral set $\Delta(X)$ may be irrational.  It
is however ``rational'' in the weak sense that its normal vectors are
contained in the quasi-lattice $\lie{X}_*(T)/(\lie{X}_*(T)\cap\n(X))$
in the quotient space $\t/\n(X)$, as we shall explain in
\ref{alinea;quasi-lattice}.
\end{alinea}

\begin{alinea}\label{alinea;prato-he}
Other antecedents of Theorem \ref{theorem;convex} can be found in the
papers of Prato \cite{prato;non-rational-symplectic} and Ishida
\cite{ishida;transverse-kaehler} and in He's thesis
\cite{he;odd-dimensional}.  These authors impose opposite versions of
our cleanness condition: Prato and Ishida deal with leafwise
transitive torus actions, while He studies certain leafwise nontangent
torus actions.  It was our attempt to unify their results that led to
this paper.  It was observed by He \cite[Ch.\ 4]{he;odd-dimensional}
that in the absence of any cleanness hypothesis the convexity of the
image may fail.  We give further counterexamples in
\ref{alinea;failure}.  However, cleanness is not necessary for
convexity to hold.  For instance, let $Z$ be a $G$-manifold and
suppose we have an equivariant surjective submersion $f\colon Z\to X$.
Then $f^*\omega$ is a presymplectic form on $Z$ and $f^*\Phi$ is a
moment map for the $G$-action on $Z$.  Clearly $Z$ has the same moment
map image as $X$.  But the action on $Z$ is rarely clean, even if the
action on $X$ is clean.
\end{alinea}

\begin{alinea}[The leaf space as a Hamiltonian space]
\label{alinea;leaf-space}
The null subgroup $N(X)$ acts trivially on the leaf space $X/\ca{F}$,
so the induced $G$-action descends to an action of the (in general
non-Hausdorff) quotient group $G/N(X)$ on the (in general
non-Hausdorff) space $X/\ca{F}$.  The moment map also descends because
of the following basic proposition.  Here we denote by
$C^{\smash\infty}_{\smash\bas}(X)$ the set of \emph{basic} smooth
functions on $X$, i.e.\ those that are constant along the leaves.  By
the \emph{affine span} of a subset $A$ of a vector space $E$ we mean
the smallest affine subspace of $E$ that contains $A$, i.e.\ the
intersection of all affine subspaces of $E$ that contain $A$.  We
denote by $F^\circ\subset E^*$ the annihilator of a linear subspace
$F$ of $E$.

\begin{subproposition}\label{proposition;affine}
\begin{enumerate}
\item\label{item;constant}
The moment map $\Phi$ is constant along the leaves of $\ca{F}$.
\item\label{item;poisson}
The moment map $\Phi$ induces a morphism of Poisson algebras
$\Phi^*\colon\g\to C^{\smash\infty}_{\smash\bas}(X)$.
\item\label{item;affine}
If $X$ is connected, the affine span of the image $\Phi(X)$ is of the
form $\lambda+\n(X)^\circ$ for some element $\lambda\in\g^*$ which is
fixed under the coadjoint $G$-action.
\end{enumerate}
\end{subproposition}

\begin{proof}
\eqref{item;constant}~Let $x\in X$, $v\in T_x\ca{F}=\ker(\omega_x)$,
and $\xi\in\g$.  Then $d\Phi^\xi(v)=\omega_x(\xi_X(x),v)=0$, so
$\Phi^\xi$ is constant on the leaf $\ca{F}(x)$ for all $\xi$.

\eqref{item;poisson}~It follows from \eqref{item;constant} that $\Phi$
pulls back smooth functions on $\g^*$ to basic functions on $X$.  That
$C_{\smash\bas}^{\smash\infty}(X)$ is a Poisson algebra is discussed
for example in \cite[\S\,2.2]{bursztyn;introduction-dirac}.  That
$\Phi^*$ preserves the Poisson bracket follows from the equivariance
of $\Phi$ as in the non-degenerate case.

\eqref{item;affine}~Let $\xi\in\g$.  Regarding $\xi$ as a linear
function on $\g^*$ we have: $\xi$ is constant on $\Phi(X)$ $\iff$
$\Phi^\xi$ is constant on $X$ $\iff$ $d\Phi^\xi=0$ $\iff$
$\xi\in\n(X)$.  Thus the affine span of $\Phi(X)$ is equal to
$\lambda+\n(X)^\circ$, where $\lambda=\Phi(x)$ for some $x\in X$.
Since the moment map is equivariant, this affine subspace contains the
coadjoint orbit of $\lambda$.  It follows from Lemma
\ref{lemma;module-affine}, applied to the $G$-module $\g^*$ and the
submodule $\n(X)^\circ$, that
$\lambda+\n(X)^\circ=\lambda_0+\n(X)^\circ$, where $\lambda_0$ is
$G$-fixed.
\end{proof}  

So if $X$ is connected we can replace $\Phi$ with $\Phi-\lambda$ to
obtain a new equivariant moment map which maps $X$ into
$\n(X)^\circ\cong(\g/\n(X))^*$ and which descends to a continuous map
$$\Phi_\ca{F}\colon X/\ca{F}\longto(\g/\n(X))^*.$$
This suggests the point of view that $\Phi_\ca{F}$ is the ``moment
map'' for a ``Hamiltonian action'' on the ``symplectic leaf space''
$X/\ca{F}$ of the ``Lie group'' $G/N(X)$, whose ``tangent Lie
algebra'' is $\g/\n(X)$, and that $\Delta(X)=\Phi_\ca{F}(X/\ca{F})\cap
C$ is the ``moment polytope'' for this action.  In the forthcoming
paper \cite{hoffman-lin-sjamaar;hamiltonian-stack} we will justify
this point of view in terms of the language of Lie groupoids or
differentiable stacks: we will ``integrate'' the foliated manifold
$(X,\ca{F})$ to a Lie groupoid $X_\bullet$ and the Lie algebra
homomorphism $\n(X)\to\g$ to a Lie $2$-group $G_\bullet$ which acts on
$X_\bullet$, and show that the moment polytope is a Morita invariant
of the $G_\bullet$-action on $X_\bullet$.

Here we point out just one manifestation of this Morita invariance.
Suppose that the null subgroup $N(X)$ admits a complement in the sense
that there exists an immersed Lie subgroup $K$ of $G$ with the
property that $KN(X)=G$ and $\lie{k}+\n(X)=\g$.  Let
$\m(X)=\lie{k}\cap\n(X)$ be the null ideal of the $K$-action on $X$.
Then the Lie $2$-group $N(X)\to G$ is Morita equivalent to the Lie
$2$-group $M(X)\to K$, where $M(X)$ is the immersed Lie subgroup of
$K$ generated by $\exp(\m(X))$.  The Lie algebras $\g/\n(X)$ and
$\lie{k}/\m(X)$ are isomorphic, so it follows immediately from
Proposition \ref{proposition;affine} that under the projection
$\g^*\to\lie{k}^*$ the $G$-moment map image of $X$ maps bijectively to
the $K$-moment map image.  We can play $G$ and $K$ off against each
other in two opposite ways.  (1)~If the $G$-action is clean, then the
$K$-action may not be clean, but even so the convexity theorem
guarantees the convexity of the $K$-moment map image.  (2)~If $G$ is
not compact, but $K$ is compact and acts cleanly, then the convexity
theorem guarantees the convexity of the $G$-moment map image.
\end{alinea}

\begin{alinea}[A foliated slice theorem]\label{alinea;transversal}
The first step towards the proof of Theorem \ref{theorem;convex} is to
construct equivariant foliation charts.  We would like to choose a
$G$-invariant transverse section to the foliation at a point $x$, but
a transverse section $Y$ has the property $T_xY\cap T_x\ca{F}=0$, and
therefore cannot be $G$-invariant unless the action is leafwise
nontangent at $x$.  Instead we do the next best thing, which is to
choose a $G$-invariant presymplectic submanifold $Y$ that is
transverse to $\ca{F}$ at $x$ and has the weaker property that
$T_xY\cap T_x\ca{F}\subset T_x(G\cdot x)$.  The slice theorem for
compact Lie group actions says that $x$ has an invariant open
neighbourhood which is equivariantly diffeomorphic to a homogeneous
vector bundle $E=G\times^{G_x}V$, where $V$ is the $G_x$-module
$T_xX/T_x(G\cdot x)$.  The following refinement of the slice theorem
states that we can single out a direct summand $V_1$ of $V$ such that
the corresponding subbundle $G\times^{G_x}V_1$ is a transversal $Y$ of
the desired type.  We write points in homogeneous bundles such as $E$
as equivalence classes $[g,v]$ of pairs $(g,v)\in G\times V$.

\begin{subtheorem}\label{theorem;transversal}
Let $x\in X$ and let $H=G_x$ be the stabilizer of $x$.  Define the
$H$-modules
$$
V=\frac{T_xX}{T_x(G\cdot x)},\quad V_0=\frac{T_x(G\cdot
  x)+T_x\ca{F}}{T_x(G\cdot x)},\quad V_1=V/V_0=\frac{T_xX}{T_x(G\cdot
  x)+T_x\ca{F}},
$$
and the corresponding $G$-homogeneous vector bundles
$$
E=G\times^HV,\qquad E_0=G\times^HV_0,\qquad E_1=E/E_0=G\times^HV_1.
$$
Choose an $H$-invariant inner product on $T_xX$ and identify $V\cong
V_0\oplus V_1$ as an orthogonal direct sum of $H$-modules and $E\cong
E_0\oplus E_1$ as an orthogonal direct sum of vector bundles.  There
exists a $G$-equivariant open embedding $\chi\colon E\to X$ which
sends $[1,0]$ to $x$ and has the following properties:
\begin{enumerate}
\item\label{item;sub}
$Y=\chi(E_1)$ is a presymplectic Hamiltonian $G$-manifold with
  presymplectic form $\omega_Y=\omega|_Y$ and moment map
  $\Phi_Y=\Phi|_Y$;
\item\label{item;tubular}
$U=\chi(E)$ is a $G$-equivariant tubular neighbourhood of $Y$ with
  tubular neighbourhood projection $p\colon U\to Y$ corresponding to
  the orthogonal projection $E\to E_1$; every fibre of $p$ is
  contained in a leaf of $\ca{F}$ and $p^{-1}(x)\cong V_0$;
\item\label{item;restrict}
$p^*\omega_Y=\omega_U$ and $p^*\Phi_Y=\Phi_U$, where
  $\omega_U=\omega|_U$ and $\Phi_U=\Phi|_U$;
\item\label{item;transverse}
$Y$ is transverse to the foliation $\ca{F}$;
\item\label{item;minimal}
$T_xY\cap T_x\ca{F}\subset T_x(G\cdot x)$.
\end{enumerate}
Now assume that the $G$-action is clean at $x$.  Then $\chi$ can be
chosen in such a way that in addition to
\eqref{item;sub}--\eqref{item;minimal} the following conditions are
satisfied:
\begin{enumerate}
\addtocounter{enumi}{5}
\item\label{item;leafwise}
the $G$-action on $Y$ is leafwise transitive;
\item\label{item;null}
the null ideal sheaves $\tilde{\n}_Y$ and $\tilde{\n}_U$ are constant
with stalk equal to $\n_{X,x}$.
\end{enumerate}
\end{subtheorem}

\begin{proof}
Choose a foliation chart $(O,\zeta)$ at $x$ in the following manner:
start with an $H$-invariant open neighbourhood $O$ of $x$ and a chart
$\breve{\zeta}\colon O\to T_xX$ centred at $x$ which satisfies
$T_x\breve{\zeta}=\id_{T_xX}$ and which maps the leaf of each $y\in O$
onto the affine subspace through $\breve{\zeta}(y)$ parallel to
$T_x\ca{F}$.  In other words,
$\breve{\zeta}(O\cap\ca{F}(y))=\breve{\zeta}(y)+T_x\ca{F}$ for all
$y$, and in particular $\breve{\zeta}(O\cap\ca{F}(x))=T_x\ca{F}$.  Let
$dh$ be the normalized Haar measure on $H$ and put
$\zeta(y)=\int_Hh\breve{\zeta}(h^{-1}y)\,dh$ for $y\in O$.  Then
$\zeta\colon O\to T_xX$ is $H$-equivariant and $T_x\zeta=\id_{T_xX}$.
For all $y\in O$ and $z\in O\cap\ca{F}(y)$ we have
$$
\zeta(z)-\zeta(y)=
\int_Hh\bigl(\breve{\zeta}(h^{-1}z)-\breve{\zeta}(h^{-1}y)\bigr)\,dh\in
T_x\ca{F},
$$
because $\breve{\zeta}(h^{-1}z)-\breve{\zeta}(h^{-1}y)\in T_x\ca{F}$
and $H$ preserves the subspace $T_x\ca{F}$ of $T_xX$.  This shows that
$\zeta(O\cap\ca{F}(y))$ is contained in the affine subspace
$\zeta(y)+T_x\ca{F}$.  After replacing $O$ with a smaller open set and
after rescaling $\zeta$ if necessary we obtain an $H$-equivariant
chart $\zeta\colon O\to T_xX$ centred at $x$ with the property that
$\zeta(O\cap\ca{F}(y))=\zeta(y)+T_x\ca{F}$ for all $y\in O$.  The
normal bundle of the orbit $G\cdot x\cong G/H$ is the homogeneous
vector bundle $E$, whose fibre $V$ we identify with the $H$-submodule
$T_x(G\cdot x)^\perp$ of $T_xX$ (the orthogonal complement of
$T_x(G\cdot x)$ with respect to the $H$-invariant inner product on
$T_xX$).  Likewise we identify
$$
V_0\cong T_x(G\cdot x)^\perp\cap T_x\ca{F}\qquad\text{and}\qquad
V_1\cong T_x(G\cdot x)^\perp\cap(T_x\ca{F})^\perp.
$$
Then $V=V_0\oplus V_1$ is an orthogonal direct sum and the map
$\chi_V=\zeta^{-1}|_V$ is an $H$-equivariant embedding of the fibre
$V$ into $O$ which maps every affine subspace of $V$ parallel to $V_0$
into a leaf of $\ca{F}$.  The map $\chi_V$ extends uniquely to a
$G$-equivariant open immersion $\chi\colon E\to X$, the image of which
is the open set $U=\chi(E)=G\cdot O$ and which maps each fibre of the
orthogonal projection $E\to E_1$ into a leaf of $\ca{F}$.  Let us
choose $O$ so small that $\chi$ is an embedding.  Then $U$ is a
tubular neighbourhood of the orbit $G\cdot x$.  We claim that
$\chi\colon E\to X$ satisfies requirements
\eqref{item;sub}--\eqref{item;minimal}.  The sum $E=E_1\oplus E_2$ is
a vector bundle over $E_1$ with fibre $V_0$, which proves
\eqref{item;tubular}.  The form $\omega$ restricts to forms $\omega_U$
on $U=\chi(E)$ and $\omega_Y$ on $Y=\chi(E_1)$.  The map $\Phi$
restricts to maps $\Phi_U\colon U\to\g^*$ and $\Phi_Y\colon Y\to\g^*$.
Since $\omega$ is $\ca{F}$-basic
(i.e.\ $\iota(v)\omega=\iota(v)d\omega=0$ for all vectors $v$ tangent
to $\ca{F}$) and each fibre of $p$ is contained in a leaf of $\ca{F}$,
we have $p^*\omega_Y=\omega_U$ and $p^*\Phi_Y=\Phi_U$, which proves
\eqref{item;restrict}.  This also implies that $\omega_Y$ is of the
same rank as $\omega$.  In particular $\omega_Y$ is of constant rank,
which proves \eqref{item;sub}, and $Y$ is transverse to $\ca{F}$,
which proves \eqref{item;transverse}.  The restricted foliation
$\ca{F}_Y=\ca{F}|_Y$ is the null foliation of $\omega_Y$.  The tangent
space to $Y$ at $x$ is $T_xY=T_x(G\cdot x)\oplus V_1$, so the tangent
space to $\ca{F}_Y$ at $x$ is the subspace of $T_xY$ given by
$$
T_x\ca{F}_Y=\ker(\omega_{Y,x})=T_xY\cap T_x\ca{F}=\bigl(T_x(G\cdot
x)\oplus V_1\bigr)\cap T_x\ca{F}.
$$
Since $V_1$ is orthogonal to both $T_x(G\cdot x)$ and $T_x\ca{F}$ this
gives
$$
T_x\ca{F}_Y=T_x(G\cdot x)\cap T_x\ca{F},
$$
which implies \eqref{item;minimal}.  Now assume the action is clean at
$x$.  Then we obtain
$$T_x\ca{F}_Y=T_x(N_{X,x}\cdot x).$$
Let us choose $O$ so small that $N_{X,x}=N_X(U)$; then $N_{X,x}\subset
N_{X,y}$ for all $y\in Y$.  The orbit $N_{X,x}\cdot y$ is an immersed
submanifold of $Y$ diffeomorphic to a homogeneous space of $N_{X,x}$.
Because $Y$ is a $G$-homogeneous vector bundle over $G\cdot x$, we
have a $G$-equivariant projection $Y\to G\cdot x$.  It follows that
$\dim(N_{X,x}\cdot y)\ge\dim(N_{X,x}\cdot x)$ for all $y\in Y$.  On
the other hand, since $N_{X,x}$ is contained in $N_{X,y}$,
$N_{X,x}\cdot y$ is contained in the leaf $\ca{F}_Y(y)$, whose
dimension is independent of $y$.  So we see that $\dim(N_{X,x}\cdot
y)=\dim(N_{X,x}\cdot x)$ and $T_y\ca{F}_Y=T_y(N_{X,x}\cdot y)$ for all
$y\in Y$.  This proves \eqref{item;leafwise}, but it actually proves
something a bit stronger, namely $T_y(G_{\bar{y}}\cdot
y)=T_y(N_{X,x}\cdot y)$ for all $y\in Y$.  This means
$\g_{\bar{y}}/\g_y=(\n_{X,x}+\g_y)/\g_y$, i.e.
\begin{equation}\label{equation;stalks}
\g_{\bar{y}}=\g_y+\n_{X,x}
\end{equation}
for all $y\in Y$.  Let $z\in U$ and put $y=p(z)\in Y$.  To prove
\eqref{item;null} it is enough to show that
$\n_{X,z}=\n_{Y,y}=\n_{X,x}$.  Let $\xi\in\g$.  By definition we have
$\xi\in\n_{X,z}$ if and only if $\iota(\xi_{U'})\omega_{U'}=0$ for
some open neighbourhood $U'\subset U$ of $z$.  Since the vector fields
$\xi_U$ and $\xi_Y$ are $p$-related and $p^*\omega_Y=\omega_U$, we
have $\iota(\xi_{U'})\omega_{U'}=0$ if and only if
$\iota(\xi_{p(U')})\omega_{p(U')}=0$, i.e.\ $\xi\in\n_{X,y}$.  This
proves $\n_{X,z}=\n_{Y,y}$.  Taking $z=x$ yields $\n_{X,x}=\n_{Y,x}$.
We finish by showing that $\n_{Y,y}=\n_{Y,x}$.  We replace $O$, if
necessary, by a smaller open neighbourhood of $x$ such that $Y$ has
the property that $\n_Y(Y)=\n_{Y,x}$.  Then we have
$\n_{Y,y}\supset\n_{Y,x}$.  Supposing our assertion
$\n_{Y,y}=\n_{Y,x}$ to be false, we can find $\xi\in\g$ such that
$\xi\in\n_y\setminus\n_x$.  This means that the vector field $\xi_Y$
is tangent to $\ca{F}_Y$ in an invariant neighbourhood $U_2$ of $y$ in
$Y$ but not tangent to $\ca{F}_Y$ in a invariant neighbourhood $U_1$
of $x$ in $Y$.  In other words,
$$
\xi_Y(w)\in T_w(G\cdot w)\cap T_w\ca{F}_Y=T_w(G_{\bar{w}}\cdot w)
$$
for all $w\in U_2$, but
$$
\xi_Y(v)\not\in T_v(G\cdot v)\cap T_v\ca{F}_Y=T_v(G_{\bar{v}}\cdot v)
$$
for some $v\in U_1$.  Because of \eqref{equation;stalks} this means
that $\xi\in\g_w+\n_{X,x}$ for all $w\in U_2$ but
$\xi\not\in\g_v+\n_{X,x}$ for some $v\in U_1$.  But, the group $G$
being compact, by choosing $w$ to be generic with respect to the
$G$-action we can arrange for $\g_w$ to be a subalgebra of $\g_v$,
which is a contradiction.  Therefore $\n_{Y,y}=\n_{Y,x}$.
\end{proof}

\begin{subcorollary}\label{corollary;clean}
\begin{enumerate}
\item\label{item;clean-x}
If the action is clean at $x\in X$, then the sheaves $\tilde{\n}$ and
$\tilde{N}$ are constant on a neighbourhood of $x$.
\item\label{item;clean-global}
Suppose that the $G$-action on $X$ is clean and that $X$ is connected.
Then the sheaves $\tilde{\n}$ and $\tilde{N}$ are constant.  It
follows that $T_x(N(X)\cdot x)=T_x(G\cdot x)\cap T_x\ca{F}$ for all
$x\in X$.
\item\label{item;leafwise-transitive-global}
If $X$ is connected and the $G$-action on $X$ is leafwise transitive,
then $\ca{F}(x)=N(X)\cdot x$ for all $x$.
\end{enumerate}
\end{subcorollary}

\begin{proof}
\eqref{item;clean-x}~is a restatement of Theorem
\ref{theorem;transversal}\eqref{item;null}.
\eqref{item;clean-global}~follows from \eqref{item;clean-x} and the
monotonicity property $\n(U_1)\supset\n(U_2)$ for $U_1\subset U_2$ of
the presheaf $\n$.  \eqref{item;leafwise-transitive-global}~follows
immediately from \eqref{item;clean-global}.
\end{proof}

If the $G$-action is leafwise nontangent at $x$, the transversal $Y$
of Theorem \ref{theorem;transversal} is a section of the foliation and
therefore symplectic, which shows that $X$ is near $x$ a
$G$-equivariant bundle over a symplectic Hamiltonian $G$-manifold.  In
particular, if $x$ is a fixed point the action is leafwise nontangent
at $x$ and we obtain the following linearization or equivariant
Darboux theorem.  In the statement of this theorem we regard the
tangent space $T_xX$ as a presymplectic manifold with constant
presymplectic form $\omega_x$.  It follows from Corollary
\ref{corollary;hessian} that at a fixed point $x$ the moment map has a
well-defined Hessian $T^2_x\Phi\colon T_xX\to\g^*$.

\begin{subcorollary}[equivariant presymplectic Darboux theorem]
\label{corollary;linear}
Let $x\in X$ be a fixed point of $G$.  Then $x$ has a $G$-invariant
neighbourhood that is isomorphic as a presymplectic Hamiltonian
$G$-manifold to a $G$-invariant neighbourhood of the origin in the
tangent space $T_xX$, equipped with the presymplectic structure
$\omega_x$ and the moment map $\lambda+T_x^2\Phi$, where
$\lambda=\Phi(x)\in(\g^*)^G$ and the Hessian is given by
$T_x^2\Phi^\xi(v)=\frac12\omega_x(\xi(v),v)$.
\end{subcorollary}  

\begin{proof}
Since $G_x=G$, we have $E=V=T_xX$, $E_0=V_0=T_x\ca{F}$, and
$E_1=V_1=T_xX/T_x\ca{F}$.  Moreover, $T_xY\cap T_x\ca{F}=0$, so
$Y=\chi(V_1)$ is symplectic.  Therefore $U=\chi(V)=\chi(V_0\times
V_1)=V_0\times Y$.  Applying the symplectic equivariant Darboux
theorem (see
\cite[Theorem~22.2]{guillemin-sternberg;techniques;;1990}) to the
symplectic Hamiltonian $G$-manifold $Y$ and the fixed point $x\in Y$,
we find that a $G$-invariant neighbourhood of $x$ in $X$ is
presymplectically and $G$-equivariantly isomorphic to a neighbourhood
of the origin in $V\cong V_0\oplus V_1$.  The action on $V$ being
linear, the moment map on $V$ is quadratic with constant term
$\lambda$ and homogeneous quadratic part $T_x^2\Phi=\Phi_V$, where
$\Phi_V$ is as in Lemma \ref{lemma;symplectic-module}.
\end{proof}  
\end{alinea}

\begin{alinea}[Symplectization]\label{alinea;coisotropic}
The second step towards the proof of Theorem \ref{theorem;convex} is
symplectization.  Let $T^*\ca{F}$ be the vector bundle dual to
$T\ca{F}$ and let $\pr\colon T^*\ca{F}\to X$ be the bundle projection.
We choose a $G$-invariant Riemannian metric on $X$.  (Recall our
standing hypothesis that $G$ is compact.)  Let $TX\to T\ca{F}$ be the
orthogonal projection onto the subbundle $T\ca{F}$ of $TX$ (with
respect to the metric) and let $j\colon T^*\ca{F}\to T^*X$ be the dual
embedding.  Let $\omega_0$ be the standard symplectic form on the
cotangent bundle $T^*X$ and let $\Omega=\pr^*\omega+j^*\omega_0$.  The
$2$-form $\Omega$ on $T^*\ca{F}$ is symplectic in a neighbourhood of
the zero section $X$ and the embedding $X\to T^*\ca{F}$ is
coisotropic.  See \cite{gotay;coisotropic-presymplectic} for these
facts.  The $G$-action on $T^*\ca{F}$ is Hamiltonian with moment map
$\Psi\colon T^*\ca{F}\to\g^*$ given by
$$\Psi=\pr^*\Phi+j^*\Phi_0,$$
where $\Phi_0\colon T^*X\to\g^*$ is the moment map for the cotangent
action given by the dual pairing $\Phi_0^\xi(y)=\inner{y,\xi_X(x)}$
for $y\in T^*_xX$.  The germ at $X$ of the Hamiltonian $G$-manifold
$T^*\ca{F}$ is called the \emph{symplectization} of $X$.  The next
result says that in the leafwise transitive case every fibre of $\Phi$
is a fibre of $\Psi$ and that the image of $\Phi$ is the intersection
of the image of $\Psi$ with an affine subspace.

\begin{subproposition}\label{proposition;symplectization}
Assume that $X$ is connected and that the $G$-action on $X$ is
leafwise transitive.  Let $\lambda\in(\g^*)^G$ be as in Proposition
\ref{proposition;affine}\eqref{item;affine}.
\begin{enumerate}
\item\label{item;preimage-affine}
$X=\Psi^{-1}(\lambda+\n(X)^\circ)$.
\item\label{item;moment-clean}
$\Psi\colon T^*\ca{F}\to\g^*$ intersects the affine subspace
  $\lambda+\n(X)^\circ$ cleanly.
\item\label{item;image-affine}
$\Phi(X)=\Psi(T^*\ca{F})\cap(\lambda+\n(X)^\circ)$.
\end{enumerate}
\end{subproposition}

\begin{proof}
\eqref{item;preimage-affine}~Let $x\in X$.  Then $\Phi_0(j(x))=0$, so
$$\Psi(x)=\Phi(x)+\Phi_0(j(x))=\Phi(x),$$
and so $\Psi(x)$ is in $\lambda+\n(X)^\circ$ by Proposition
\ref{proposition;affine}\eqref{item;affine}.  This shows that
$X\subset\Psi^{-1}(\lambda+\n(X)^\circ)$.  Conversely, let $z\in
T^*\ca{F}$ and suppose that $\Psi(z)\in \lambda+\n(X)^\circ$.  Put
$x=\pr(z)\in X$ and $y=j(z)\in T_x^*X$.  Then
$\Psi(z)=\Phi(x)+\Phi_0(y)$, so
$\Phi_0(y)=\Psi(z)-\Phi(x)\in\n(X)^\circ$.  It follows that
$\inner{y,\xi_X(x)}=0$ for every $\xi\in\n(X)$.  In other words, $y\in
T_x^*X$ annihilates the tangent space to the $N(X)$-orbit $N(X)\cdot
x$.  By Corollary \ref{corollary;clean}\eqref{item;clean-global} we
have $N(X)\cdot x=\ca{F}(x)$ because the action is leafwise
transitive.  Therefore $y$ annihilates all of $T_x\ca{F}$.  But $y$ is
in the image of $j$, which is a splitting of the natural surjection
$T^*X\to T^*\ca{F}$, and therefore
$y\in\im(j)\cap(T_x\ca{F})^\circ=0$.  We conclude that $z=x\in X$.

\eqref{item;moment-clean}~We have just shown that
$\Psi^{-1}(\lambda+\n(X)^\circ)$ is equal to $X$ and is therefore a
submanifold of $M=T^*\ca{F}$.  It remains only to show that $X$ has
the correct tangent bundle, namely $TX=(T\Psi)^{-1}(\n(X)^\circ)$.
Let $x\in X$.  Regarding $x$ as a point in the zero section of $T^*X$,
we have $T_x(T^*X)=T_xX\oplus T^*_xX$ and
$$
T_xM=T_xX\oplus T^*_x\ca{F},\qquad T_x\Psi=\pr^*T_x\Phi+j^*T_x\Phi_0,
$$
where now $\pr$ stands for the projection $T_xM\to T_xX$ and $j$ for
the inclusion $T_xM\to T_x(T^*X)$.  The derivative at $x$ of
$\Phi_0\colon T^*X\to\g^*$ is the linear map $T_x\Phi_0\colon
T_xX\oplus T^*_xX\to\g^*$ given by
$\inner{T_x\Phi_0(u,v),\xi}=\inner{v,\xi_X(x)}$ for $u\in T_xX$, $v\in
T^*_xX$ and $\xi\in\g$.  Now let $w\in T_xM$ and suppose
$T_x\Psi(w)\in\n(X)^\circ$.  Put $u=\pr(w)\in T_xX$ and $v=j(w)\in
T_x(T^*X)$.  Then $T_x\Phi_0(v)=T_x\Psi(w)-T_x\Phi(u)\in\n(X)^\circ$,
so $\inner{v,\xi_X(x)}=0$ for all $\xi\in\n(X)$.  As in the proof of
\eqref{item;preimage-affine} we deduce from this that $v=0$,
i.e.\ $w=u\in T_xX$.  This proves $(T_x\Psi)^{-1}(\n(X)^\circ)\subset
T_xX$.  The reverse inclusion follows from the fact that $X$ is
contained in $\Psi^{-1}(\lambda+\n(X)^\circ)$.

\eqref{item;image-affine}~follows immediately from
\eqref{item;preimage-affine}.
\end{proof}

The next result, which is essentially due to Guillemin and Sternberg,
is a partial converse to Proposition \ref{proposition;symplectization}
as well as a useful source of examples.

\begin{subproposition}\label{proposition;coisotropic}
Let $(M,\omega_M)$ be a symplectic Hamiltonian $G$-manifold with
moment map $\Phi_M\colon M\to\g^*$.  Let $\lambda\in\g^*$ and let $\a$
be an ideal of $\g$ satisfying $\lie{k}\subset\a\subset\g_\lambda$,
where $\lie{k}$ is the kernel of the infinitesimal action
$\g\to\Gamma(TM)$.  Assume that $\Phi_M$ intersects the affine
subspace $\lambda+\a^\circ$ cleanly.  Then
$X=\Phi_M^{-1}(\lambda+\a^\circ)$ is a coisotropic submanifold of $M$
preserved by the action of $G$.  Therefore $X$ is a presymplectic
Hamiltonian $G$-manifold with presymplectic form $\omega=\omega_M|_X$
and moment map $\Phi=\Phi_M|_X$.  The action on $X$ is leafwise
transitive and the null ideal $\n(X)$ of $X$ is equal to $\a$.
\end{subproposition}

\begin{proof}
It follows from Lemma \ref{lemma;coadjoint} that $\lambda+\a^\circ$ is
preserved by the coadjoint action.  Therefore $X$ is preserved by $G$
and, by \cite[Theorem 26.4]{guillemin-sternberg;techniques;;1990}, $X$
is coisotropic and the leaves of the null foliation of the
presymplectic form $\omega$ are the orbits of the $A$-action on $X$,
where $A$ is the connected immersed normal subgroup corresponding to
the ideal $\a$.  Hence the action is leafwise transitive.  By
Proposition \ref{proposition;affine}\eqref{item;affine} the affine
span of $\Phi(X)$ is the affine subspace $\lambda+\n(X)^\circ$ and the
affine span of $\Phi(M)$ is $\lambda+\lie{k}^\circ$.  Therefore
$$
\lambda+\n(X)^\circ=(\lambda+\a^\circ)\cap(\lambda+\lie{k}^\circ)=
\lambda+\a^\circ.
$$
We conclude that $\n(X)=\a$.
\end{proof}  

\end{alinea}

\begin{alinea}[Proof of the convexity theorem]\label{alinea;proof}
First we prove the following local version of the presymplectic
convexity theorem.  Recall that $T$ denotes a maximal torus of $G$,
$\t^*$ the dual of its Lie algebra $\t$, and $C$ a closed chamber in
$\t^*$.  Recall also that every coadjoint orbit intersects the chamber
$C$ in exactly one point and that the inclusion $C\to\g^*$ induces a
homeomorphism $C\to\g^*/\Ad^*(G)$, the quotient of $\g^*$ by the
coadjoint action.  (See
e.g.\ \cite[\S~IX.5.2]{bourbaki;groupes-algebres}.)  We identify
$\g^*/\Ad^*(G)$ with $C$ via this homeomorphism and denote by
$\phi\colon X\to C$ the composition of the moment map $\Phi\colon
X\to\g^*$ with the quotient map $\g^*\to C$.  Then the intersection
$\Delta(X)=\Phi(X)\cap C$ is nothing but the image $\phi(X)$.

\begin{subtheorem}[local presymplectic convexity theorem]
\label{theorem;local-convex}
Assume that the $G$-action is clean at $x\in X$.  Then there exist a
rational convex polyhedral cone $\Delta_x$ in $\t^*$ with apex
$\phi(x)$ and a basis of $G$-invariant open neighbourhoods $U$ of $x$
in $X$ with the following properties:
\begin{enumerate}
\item\label{item;connected}
the fibres of the map $\phi|_U$ are connected;
\item\label{item;cone}
$\phi\colon U\to\Delta_x\cap\bigl(\phi(x)+\n_x^\circ\bigr)$ is an open
  map.
\end{enumerate}  
\end{subtheorem}

\begin{proof}
Let us replace $x$ by a suitable $G$-translate in order that
$\Phi(x)=\phi(x)$.  Choose a transversal $Y$ at $x$ and a
$G$-invariant tubular neighbourhood $U$ of $Y$ as in Theorem
\ref{theorem;transversal}.  Since $G$ is connected, we may assume $Y$
and $U$ to be connected.  Then $\Phi(U)\subset\phi(x)+\n_x^\circ$ by
Proposition \ref{proposition;affine}\eqref{item;affine}, and therefore
$\phi(U)\subset\phi(x)+\n_x^\circ$.  Let $\phi_U=\phi|_U$ and
$\phi_Y=\phi|_Y$.  Let $M$ be the symplectization of $Y$ as defined in
\ref{alinea;coisotropic}.  Let $\Psi\colon M\to\g^*$ the moment map
for the $G$-action on $M$ and $\psi\colon M\to C$ the composition of
$\Psi$ with the orbit map $\g^*\to C$.  It follows from Theorem
\ref{theorem;transversal} and Proposition
\ref{proposition;symplectization} that for all $\nu\in\phi(U)$
\begin{equation}\label{equation;fibre}
\phi_U^{-1}(\nu)=p^{-1}(\phi_Y^{-1}(\nu)),\qquad
\phi_Y^{-1}(\nu)=\psi^{-1}(\nu),
\end{equation}
and that
\begin{equation}\label{equation;image}
\phi_U=\phi_Y\circ p,\qquad\phi_Y=\psi|_Y,\qquad
Y=\psi^{-1}\bigl(\phi(x)+\n_x^\circ\bigr).
\end{equation}
The local convexity theorem in the symplectic case (see
e.g.\ \cite[Theorem~6.5]{sjamaar;convexity}) states that the fibres of
$\psi$ are connected and that $\psi\colon M\to\Delta_x$ is an open
mapping to a rational convex polyhedral cone $\Delta_x$ in $\t^*$ with
apex $\phi(x)=\psi(x)$.  Taking this into account, we see from
\eqref{equation;fibre} that the fibres of $\phi_U$ contract onto
fibres of $\psi$ and are therefore connected as well, and we see from
\eqref{equation;image} that $\phi_U$ is an open map to
$\Delta_x\cap\bigl(\phi(x)+\n_x^\circ\bigr)$.
\end{proof}

\begin{proof}[Proof of Theorem \ref{theorem;convex}]
Theorem \ref{theorem;local-convex} asserts that the family of convex
cones $\Delta_x\cap\bigl(\phi(x)+\n_x^\circ\bigr)$, where $x$ ranges
over $X$, is a system of \emph{local convexity data} for the quotient
map $\phi$ in the sense of Hilgert, Neeb and Planck
\cite[Definition~3.3]{hilgert-neeb-plank;symplectic-convexity;compositio}.
Parts \eqref{item;open} and \eqref{item;convex} of Theorem
\ref{theorem;convex} now follow from the local-to-global principle due
to these authors,
\cite[Theorem~3.10]{hilgert-neeb-plank;symplectic-convexity;compositio}.
It remains to prove part \eqref{item;rational}.  The local-to-global
principle also tells us that for each $x\in X$ the cone
$\Delta_x\cap\bigl(\phi(x)+\n_x^\circ\bigr)$ is equal to the
intersection of all supporting half-spaces of the convex set
$\Delta(X)$ at the point $\phi(x)$.  A closed convex set is equal to
the intersection of all its supporting half-spaces, and therefore
$$
\Delta(X)=\bigcap_{x\in X}\Delta_x\cap\bigl(\phi(x)+\n_x^\circ\bigr).
$$
Since the action is clean, by Corollary
\ref{corollary;clean}\eqref{item;clean-global} we have $\n_x=\n(X)$
for all $x$.  By Proposition
\ref{proposition;affine}\eqref{item;affine} there is an
$\Ad^*(G)$-fixed $\lambda\in\g^*$ such that
$\Phi(X)\subset\lambda+\n(X)^\circ=\phi(x)+\n(X)^\circ$ for all $x$.
After shifting the moment map we may assume $\lambda=0$.  Then
\begin{equation}\label{equation;cones}
\Delta(X)=\bigcap_{x\in X}\Delta_x\cap\n(X)^\circ,
\end{equation}
the intersection of the locally finite family of rational cones
$\Delta_x$ with the linear subspace $\n(X)^\circ$.  Now suppose that
the null subgroup $N(X)$ is closed.  Then by Corollary
\ref{corollary;closed-rational} the subspace $\n(X)\cap\t$ of $\t$ is
rational and therefore the subspace $\n(X)^\circ\cap\t^*$ of $\t^*$ is
rational.  Because of this and \eqref{equation;cones} the polyhedral
set $\Delta(X)$ is rational.  Conversely, suppose that $\Delta(X)$ is
rational.  Let $\lie{z}$ be the centre of $\g$ and
$\pr\colon\g^*\to\lie{z}^*$ the projection dual to the inclusion of
$\lie{z}$ into $\g\cong[\g,\g]\oplus\lie{z}$.  Then the image
$\pr(\Delta(X))$ is a rational polyhedral subset of $\lie{z}$.  The
chamber $C$ of $\g^*$ is of the form $C=C'\times\lie{z}^*$, where $C'$
is a chamber of $[\g,\g]^*$, so we see that $\pr(\Delta(X))$ is equal
to the image of $X$ under the moment map $\Phi_Z$ for the action of
the central subtorus $Z=\exp_G(\lie{z})$ of $G$.  By Proposition
\ref{proposition;affine}\eqref{item;affine} the affine span of
$\Phi_Z(X)=\pr(\Delta(X))$ is equal to $(\n(X)\cap\lie{z})^\circ$.  It
follows that $\n(X)\cap\lie{z}$ is a rational subspace of $\lie{z}$.
By Corollary \ref{corollary;closed-rational} we conclude that $N(X)$
is a closed subgroup of $G$.
\end{proof}

\end{alinea}  

\begin{alinea}[Irrationality]\label{alinea;quasi-lattice}
The  following result is a consequence of the proof of Theorem
\ref{theorem;convex}\eqref{item;rational}.

\begin{subcorollary}
Under the hypotheses of Theorem \ref{theorem;convex}, the polyhedral
set $\Delta(X)$ is rational if and only if the moment map image for
the action of the central subtorus $Z=\exp(\lie{z})$ is a rational
polyhedral subset of $\lie{z}^*$.  In particular $\Delta(X)$ is always
rational if $G$ is semisimple.
\end{subcorollary}

So we see that the irrationality of presymplectic moment polytopes is
essentially an abelian phenomenon.

As explained in \ref{alinea;leaf-space}, $\Delta(X)$ is best regarded
intrinsically as a subset of $\t_0^*$, where $\t_0$ is the quotient
$\t/(\t\cap\n(X))$.  The vector spaces $\t_0$ and $\t_0^*$ have no
natural $\Q$-structure (except when the null subgroup $N(X)$ is
closed), but as a substitute we have the quasi-lattice
$\Lambda=\im\bigl(\lie{X}_*(T)\to\t_0\bigr)$.  Here
$\lie{X}_*(T)=\ker(\exp\colon\t\to T)$ is the exponential lattice of
the maximal torus, and by a \emph{quasi-lattice} in a vector space $V$
we mean a finitely generated additive subgroup that spans $V$ over
$\R$.  The rank of the quasi-lattice $\Lambda$ is $\ge\dim(\t_0)$,
where equality holds if and only if $\t\cap\n(X)$ is rational,
i.e.\ $N(X)$ is closed.  If we regard the polyhedral set $\Delta(X)$
as a subset of $\t_0^*$, the normal vectors to its facets are in
$\t_{\smash0}^{\smash**}=\t_0$.  It follows from
\eqref{equation;cones} that these normal vectors are contained in the
quasi-lattice $\Lambda$.

\begin{subcorollary}
Under the hypotheses of Theorem \ref{theorem;convex}, the polyhedral
subset $\Delta(X)\subset\t_0^*$ is the intersection of a locally
finite collection of halfspaces of the form $\inner{\eta,\cdot}\ge a$,
where $a\in\R$ and $\eta$ is in the quasi-lattice
$\Lambda=\im\bigl(\lie{X}_*(T)\to\t_0\bigr)$.
\end{subcorollary}

\end{alinea}

%%%%%%%%%%%%%%%%%%%%%%%%%%%%%%%%%%%%%%%%%%%%%%%%%%%%%%%%%%%%%%%%%%%%%%%%
\section{Morse functions and the abelian case}\label{section;morse}
%%%%%%%%%%%%%%%%%%%%%%%%%%%%%%%%%%%%%%%%%%%%%%%%%%%%%%%%%%%%%%%%%%%%%%%%

\begin{alinea}
This section is a discussion of Morse-theoretic properties of
presymplectic moment maps and of the presymplectic convexity theorem
in the abelian case.  The main result is Theorem
\ref{theorem;morse-bott}, which asserts that the components of a
presymplectic moment map are Morse-Bott functions under the assumption
that the action is clean.

We keep the notational conventions of Section \ref{section;convex}:
$X$ is a manifold with presymplectic form $\omega$ and $G$ is a
compact connected Lie group acting on $M$ in a Hamiltonian fashion
with moment map $\Phi$.  We denote the null foliation of $\omega$ by
$\ca{F}$, the null ideal of an open subset $U$ by $\n(U)$ and the null
subgroup by $N(U)$.  We denote the leaf of $x\in X$ by $\ca{F}(x)$.
When we think of the leaf as a point in the leaf space $X/\ca{F}$ we
denote it by $\bar{x}$.
\end{alinea}

\begin{alinea}\label{alinea;compatible}
Let $E$ be a finite-dimensional real vector space and $\sigma$ a
presymplectic form on $E$.  We call an inner product on $E$
\emph{compatible} with $\sigma$ if on the linear subspace
$F=\ker(\sigma)^\perp$ orthogonal to $\ker(\sigma)$ the inner product
is compatible with the symplectic form $\sigma|_F$ in the usual sense,
namely that the endomorphism $J$ of $F$ determined by
$\inner{u,v}=\sigma(Ju,v)$ for all $u$, $v\in F$ defines an orthogonal
complex structure.  On the subspace $F$ the symplectic form and the
compatible inner product combine to give a Hermitian inner product.
Compatible inner products always exist and, if $E$ is a presymplectic
$H$-module for some compact Lie group $H$, can be chosen to be
$H$-invariant.  A choice of such an inner product makes the subspace
$F$ a unitary $H$-module.
\end{alinea}

\begin{alinea}\label{alinea;critical}
We say that a subset $Z$ of $X$ is a \emph{submanifold at $x\in X$} if
$x$ has an open neighbourhood $U$ with the property that $Z\cap U$ is
a submanifold of $U$.  Let $f\colon X\to\R$ be smooth and let
$\Gamma_f=\{\,x\in X\mid T_xf=0\,\}$ be its critical set.  We say a
critical point $x\in\Gamma_f$ is \emph{nondegenerate in the sense of
  Bott} if $\Gamma_f$ is a submanifold at $x$ and the kernel of the
Hessian $T_x^2f\colon T_xX\to\R$ is equal to the tangent space
$T_x\Gamma_f$.  The \emph{index} of $f$ at $x$ is the dimension of a
maximal negative definite subspace for the quadratic form $T_x^2f$.
We say $f$ is a \emph{Morse-Bott function} if all its critical points
are nondegenerate in the sense of Bott.
\end{alinea}

\begin{alinea}\label{alinea;bott}
For $\xi\in\g$ we denote the critical set $\Gamma_{\Phi^\xi}$ of the
function $\Phi^\xi$ by $X^{[\xi]}$.  For a subalgebra $\h$ of $\g$ we
define the \emph{critical set of} $\h$ to be
$X^{[\h]}=\bigcap_{\xi\in\h}X^{[\xi]}$, the common critical set of the
functions $\Phi^\xi$ with $\xi\in\h$.  In the symplectic case the
critical set of $\h$ is the fixed point manifold of the subgroup
generated by $\h$.  The next lemma says that for this to remain true
in the presymplectic case we have to replace fixed points by fixed
\emph{leaves}.

\begin{sublemma}\label{lemma;fixed}
\begin{enumerate}
\item\label{item;fixed-leaf}
Let $\h$ be a Lie subalgebra of $\g$ and $H$ the connected immersed
subgroup of $G$ with Lie algebra $\h$.  Then
$$
X^{[\h]}=\{\,x\in X\mid\h\subset\g_{\bar{x}}\,\}=\{\,x\in X\mid
H\cdot\bar{x}=\bar{x}\,\}.
$$
\item\label{item;critical}
$X^{[\xi]}=\{\,x\in X\mid\xi\in\g_{\bar{x}}\,\}=\{\,x\in
  X\mid\xi_X(x)\in T_x\ca{F}\,\}$ for all $\xi\in\g$.
\end{enumerate}
\end{sublemma}

\begin{proof}
\eqref{item;fixed-leaf}~Let $x\in X$.  Since $H$ is connected, we have
$H\cdot\bar{x}=\bar{x}$ $\iff$ $H\subset G_{\bar{x}}$ $\iff$
$\h\subset\g_{\bar{x}}$, which proves the second equality.  Moreover,
$H\cdot\bar{x}=\bar{x}$ $\iff$ $T_x(H\cdot x)\subset T_x\ca{F}$.
Applying Lemma \ref{lemma;image-kernel}\eqref{item;kernel} to the
subgroup $H$ we see that $T_x(H\cdot x)\subset T_x\ca{F}$ is
equivalent to $T_x\Phi^\xi=0$ for all $\xi\in\h$, which proves the
first equality.

\eqref{item;critical}~follows from \eqref{item;fixed-leaf} applied to
the Lie subalgebra spanned by $\xi$.
\end{proof}  

The critical set $X^{[\xi]}$ is unaffected if we perturb $\xi$ in the
direction of the null ideal.  Specifically, if $U$ is an open subset
and $\zeta$ is in the null ideal $\n(U)$, then $d\Phi^\zeta=0$ on $U$.
Therefore
\begin{equation}\label{equation;critical}
X^{[\xi]}\cap U=X^{[\eta]}\cap U\quad\text{and}\quad
T_x^2\Phi^\xi=T_x^2\Phi^\eta
\end{equation}
for all $x\in U$ and for all $\xi$, $\eta\in\g$ satisfying
$\xi-\eta\in\n(U)$.  Similarly,
\begin{equation}\label{equation;critical-algebra}
X^{[\h]}\cap U=X^{[\h+\n(U)]}\cap U
\end{equation}
for all subalgebras $\h$.  The critical set of $\h$ is preserved by
the subgroup $HN(X)$ generated by the Lie subalgebra $\h+\n(X)$.

The next assertion implies that if the $G$-action is clean the
critical set is a submanifold and its normal bundle is symplectic.
This can easily be false without the cleanness assumption.  (See
\ref{alinea;failure} for counterexamples.)

\begin{subproposition}\label{proposition;critical}
Let $\h$ be a Lie subalgebra of $\g$ and let $x\in X^{[\h]}$.  Assume
that the $G$-action on $X$ is clean at $x$.  Then $X^{[\h]}$ is a
submanifold at $x$.  The subspace of $T_xX$ orthogonal to $T_xX^{[\h]}$
with respect to a $G_x$-invariant compatible inner product on $T_xX$
is symplectic.
\end{subproposition}

\begin{proof}
By Lemmas \ref{lemma;clean-x}\eqref{item;stalk} and
\ref{lemma;fixed}\eqref{item;fixed-leaf} the cleanness assumption
implies that $\h\subset\g_{\bar{x}}=\g_x+\n_x$.  That is to say,
$\h+\n_x=\lie{f}+\n_x$, where $\lie{f}$ is the subalgebra
$(\h+\n_x)\cap\g_x$ of $\g_x$.  Choose an open neighbourhood $U$ of
$x$ with $\n(U)=\n_x$.  Then
$$
X^{[\h]}\cap U=X^{[\h+\n_x]}\cap U=X^{[\lie{f}+\n_x]}\cap
U=X^{[\lie{f}]}\cap U
$$
by \eqref{equation;critical-algebra}, so we may just as well replace
$\h$ with $\lie{f}$.  The advantage of the subalgebra $\lie{f}$ is
that it fixes $x$ and therefore acts linearly in a $G_x$-equivariant
Darboux chart centred at $x$.  More precisely, the equivariant Darboux
Theorem, Corollary \ref{corollary;linear}, allows us to replace $X$
with the presymplectic $G_x$-module $E=T_xX$ and the functions
$\Phi^\eta$ for $\eta\in\lie{f}$ with the quadratic forms
$T_x^2\Phi^\eta$.  Choosing a $G_x$-invariant compatible inner product
on $E$, we have that $E=E_0\oplus E_1$ is an orthogonal direct sum of
a $G_x$-module $E_0$ and a unitary $G_x$-module $E_1$.  The critical
set is then $X^{[\lie{f}]}=E_0\oplus E_1^{\lie{f}}$, where
$E_1^{\lie{f}}$ denotes the $\lie{f}$-fixed subspace of $E_1$.  This
shows that $X^{[\lie{f}]}$ is a submanifold.  The orthogonal
complement of $X^{[\lie{f}]}$ is a unitary submodule of $E_1$, and in
particular it is symplectic.
\end{proof}  

Taking $\h=\R\xi$ gives the following result.

\begin{subtheorem}
Let $\xi\in\g$ and let $x\in X^{[\xi]}$ be a critical point of
$\Phi^\xi$.  Assume that the $G$-action on $X$ is clean at $x$.  Then
$x$ is nondegenerate in the sense of Bott.  Choose a $G_x$-invariant
compatible inner product on $T_xX$.  Then the positive and negative
subspaces of $T_x^2\Phi^\xi$ are symplectic subspaces of $T_xX$.  In
particular the index of $\Phi^\xi$ at $x$ is even.
\end{subtheorem}

\begin{proof}
It follows from Proposition \ref{proposition;critical} that
$X^{[\xi]}$ is a submanifold at $x$.  We argue nondegeneracy by
writing $\xi=\eta+\zeta$ with $\eta\in\g_x$ and $\zeta\in\n_x$.  Then
$X^{[\xi]}=X^{[\eta]}$ near $x$ and $T_x^2\Phi^\xi=T_x^2\Phi^\eta$
because of \eqref{equation;critical}, and the vector field $\eta_X$ is
linear in an equivariant Darboux chart at $x$.  With the same notation
as in the proof of the proposition, the subspace of $T_xX$ orthogonal
to $X^{[\eta]}$ is the sum of the nonzero weight spaces of the unitary
module $E_1$.  In other words, $(T_xX^{[\eta]})^\perp=E_1^+\oplus
E_1^-$, where $E_1^+$, resp.\ $E_1^-$, is spanned by all positive,
resp.\ negative, weight vectors, i.e.\ vectors $e\in E_1$ satisfying
$\xi(e)=\sqrt{-1}\alpha e$ for some $\alpha>0$, resp.\ $\alpha<0$.  On
$E_1^+$ the Hessian of $\Phi^\eta$ is positive definite, on $E_1^-$ it
is negative definite.
\end{proof}  

Assuming cleanness at all points of $X$ we obtain the next statement,
which is the main result of this section.

\begin{subtheorem}\label{theorem;morse-bott}
Assume that the $G$-action on $X$ is clean.  Then for every $\xi\in\g$
the component $\Phi^\xi$ of the moment map is a Morse-Bott function.
The positive and negative normal bundles of the critical set
$X^{[\xi]}$, taken with respect to a $G$-invariant compatible
Riemannian metric on $X$, are symplectic subbundles of $TX$ orthogonal
to the subbundle $T\ca{F}$.
\end{subtheorem}

In the symplectic case this result goes back to Atiyah
\cite{atiyah;convexity-commuting} and Guillemin and Sternberg
\cite{guillemin-sternberg;convexity;;1982}.  For leafwise transitive
Hamiltonian circle actions on K-contact manifolds the result was
proved by Rukimbira \cite{rukimbira;contact-minimal-characteristics}.
His result was extended to leafwise transitive presymplectic
Hamiltonian torus actions by Ishida \cite{ishida;transverse-kaehler}.
\end{alinea}

\begin{alinea}\label{alinea;abelian}
If $G$ is a torus and $X$ is compact symplectic, then the vertices of
the moment polytope $\Phi(X)$ are images of $G$-fixed points.  In the
presymplectic case there may not be any fixed points.  Instead one
needs to consider the critical points of the moment map, or
equivalently the $G$-fixed \emph{leaves}.  Moreover, we can weaken the
assumption that $G$ is abelian to the assumption that the quotient
$G/N(X)$ is abelian, or equivalently that the null ideal $\n(X)$
contains the derived subalgebra of $\g$.  By Proposition
\ref{proposition;affine}, then the moment map image is contained in
$\lie{z}(\g)^*$, the dual of the centre $\lie{z}(\g)$ of $\g$, and we
have $\Delta(X)=\Phi(X)$.

\begin{subtheorem}[abelian convexity]\label{theorem;vertex}
Assume that the null ideal $\n(X)$ contains the derived subalgebra
$[\g,\g]$, that $X$ is compact, and that the $G$-action on $X$ is
clean.  Then $\Phi\bigl(X^{[\g]}\bigr)$ is a finite subset of $\g^*$
and $\Phi(X)$ is the convex hull of $\Phi\bigl(X^{[\g]}\bigr)$.  For
every vertex $\lambda$ of $\Phi(X)$ the fibre $\Phi^{-1}(\lambda)$ is
a connected component of $X^{[\g]}$.
\end{subtheorem}

\begin{proof}
It follows from Theorem \ref{theorem;convex} that $\Phi(X)$ is a
convex polytope and therefore equal to the convex hull of its
vertices.  It follows from Proposition \ref{proposition;critical} that
$X^{[\g]}$ is a closed submanifold of $X$, and therefore has a finite
number of connected components.  The moment map is constant on each
component, so $\Phi(X^{[\g]})$ is finite.  Let $\lambda$ be a vertex
of $\Phi(X)$ and let $x\in\Phi^{-1}(\lambda)$.  Then there exists an
open subset $\Xi$ of $\t$ with the property that for every $\xi\in\Xi$
the function $\Phi^\xi$ attains its global minimum at $x$.  Hence
$T_x\Phi^\xi=0$ and $T_x^2\Phi^\xi$ is positive semidefinite for all
$\xi\in\Xi$.  Because $\Xi$ spans $\t$, this implies that $T_x\Phi=0$,
i.e.\ $x\in X^{[\g]}$.  By Theorem \ref{theorem;morse-bott}, for all
$\xi\in\Xi$ the Hessian of $\Phi^\xi$ at $x$ is positive definite in
the direction normal to $X^{[\g]}$.  Computing in an equivariant
Darboux chart $U$ at $x$ we see that the portion
$U\cap\Phi^{-1}(\lambda)$ of the fibre $\Phi^{-1}(\lambda)$ is
contained in $X^{[\g]}$.  Therefore the entire fibre is contained in
$X^{[\g]}$.  Since the fibre is connected (Theorem
\ref{theorem;convex}), it is equal to a component of $X^{[\g]}$.
\end{proof}

\end{alinea}

%%%%%%%%%%%%%%%%%%%%%%%%%%%%%%%%%%%%%%%%%%%%%%%%%%%%%%%%%%%%%%%%%%%%%%%%
\section{Examples}\label{section;example}
%%%%%%%%%%%%%%%%%%%%%%%%%%%%%%%%%%%%%%%%%%%%%%%%%%%%%%%%%%%%%%%%%%%%%%%%

\begin{alinea}
In this section $G$ denotes a compact connected Lie group, $\t$ a
Cartan subalgebra of $\g$ and $C$ a chamber in $\t^*$.
\end{alinea}

\begin{alinea}[Failure of convexity]\label{alinea;failure}
Z.~He \cite[Ch.\ 4]{he;odd-dimensional} gave the first example of a
presymplectic Hamiltonian torus action with a nonconvex moment map
image.  Here we show that such examples are ubiquitous.  In particular
presymplectic Hamiltonian actions are typically not clean.  Our
starting point is the following elementary fact, which is implicit in
\cite[Theorem 26.4]{guillemin-sternberg;techniques;;1990}.  (Cf.\ also
Proposition \ref{proposition;coisotropic}.)

\begin{sublemma}
Let $(M,\pi_M)$ and $(N,\pi_N)$ be Poisson manifolds and $\phi\colon M\to
N$ a Poisson morphism.  Let $Y$ be a coisotropic submanifold of $N$
that intersects $\phi$ cleanly.  Then $X=\phi^{-1}(Y)$ is a coisotropic
submanifold of $M$.
\end{sublemma}  

\begin{proof}
That $\phi$ is a Poisson morphism means by definition that the square
$$
\xymatrix@H=1.5em{
TM\ar[r]^{T\phi}&\phi^*TN
\\
T^*M\ar[u]^{\pi_M^\sharp}&\phi^*T^*N\ar[l]_{T^*\phi}\ar[u]_{\phi^*\pi_N^\sharp}
}
$$
commutes.  That $Y$ is coisotropic means that $\pi_N^\sharp(T^\circ
Y)$ is contained in $TY$, where $T^\circ Y$ denotes the annihilator of
$TY$ in $T^*M|_Y$.  That $Y$ intersects $\phi$ cleanly means that $X$
is a submanifold with tangent bundle $TX=(T\phi)^{-1}(\phi^*TY)$.
Therefore the annihilator of $TX$ is $T^\circ X=T^*\phi(\phi^*T^\circ
Y)$.  From $\pi_N^\sharp(T^\circ Y)\subset TY$ we infer
$$
T\phi\circ\pi_M^\sharp\circ T^*\phi(\phi^*T^\circ Y)\subset\phi^*TY,
$$
and hence $T\phi(\pi_M^\sharp(T^\circ X))\subset \phi^*TY$.  We
conclude that $\pi_M^\sharp(T^\circ X)\subset TX$.
\end{proof}

Let us apply this result to a symplectic Hamiltonian $G$-manifold $M$
with symplectic form $\omega_M$, taking $N$ to be the linear Poisson
manifold $\g^*$ and $\phi$ to be the moment map.  Any $G$-invariant
submanifold $Y$ of $\g^*$ is coisotropic.  Therefore, if $Y$
intersects $\phi$ cleanly, its preimage $X=\phi^{-1}(Y)$ is a
coisotropic submanifold of $M$.  It follows that the closed $2$-form
$\omega=\omega_M|_X$ has constant corank equal to the codimension of
$X$ in $M$.  The moment map is equivariant, so $X$ is preserved by the
$G$-action, and the $G$-action on $X$ is Hamiltonian with moment map
$\Phi=\phi|_M$.  Thus $X$ is a presymplectic Hamiltonian $G$-manifold.
Its moment map image is
$$\Phi(X)=\phi(\phi^{-1}(Y))=Y\cap\phi(M).$$
It is easy to choose $Y$ in such a manner that $\Delta(X)=\Phi(X)\cap
C$ is not convex.

For a specific example let $\T=\R/2\pi\Z$ be the circle and let
$G=\T^d$ be the $d$-torus acting on $M=\C^d$ in the standard way,
$$
t\cdot x=(t_1,t_2,\dots,t_d)\cdot(x_1,x_2,\dots,x_d)=
\bigl(e^{it_1}x_1,e^{it_2}x_2,\dots,e^{it_d}x_d\bigr).
$$
Then $\g=\R^d$ and $\g^*=(\R^d)^*\cong\R^d$, and the map
$\phi\colon\C^d\to\R^d$ defined by
$\phi(x)=\frac12\bigl(\abs{x_1}^2,\abs{x_2}^2,\dots,\abs{x_d}^2\bigr)$
is a moment map for this action with respect to the standard
symplectic form $\omega_0=\frac1{2i}\sum_jdx_j\wedge d\bar{x}_j$.  Let
$Y$ be a submanifold of $\R^d$ of codimension $k$ which is transverse
to the faces of the orthant $\R^d_{\ge0}$.  Then $Y$ is transverse to
$\phi$, so $X=\phi^{-1}(Y)$ is a submanifold of $\C^d$ of (real)
codimension $k$.  Obviously $Y$ is coisotropic with respect to the
zero Poisson structure, so $X$ is a presymplectic Hamiltonian
$G$-manifold with presymplectic form of corank $k$.  Its moment map
image is the intersection of $Y$ with the orthant,
\begin{equation}\label{equation;orthant}
\Phi(X)=\phi(\phi^{-1}(Y))=Y\cap\R^d_{\ge0},
\end{equation}
which is a $d-k$-manifold with corners, but is of course seldom
convex.

For instance, the following curve in the positive quadrant ($d=2$,
$k=1$) is the moment map image of a $\T^2$-action on a presymplectic
$3$-sphere.

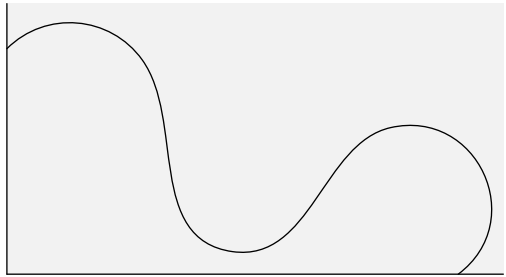
\begin{center}
\includegraphics{nonconvex-1.mps}
\end{center}

This class of examples displays some other phenomena of interest, such
as the existence of nontrivial deformations of presymplectic
Hamiltonian actions.  The equivariant Darboux theorem implies that
symplectic Hamiltonian $G$-manifolds cannot be continuously deformed
locally near any point.  The Moser stability theorem implies that the
same is true globally for compact symplectic Hamiltonian $G$-manifolds
as long as we move the symplectic form within a fixed cohomology
class.  We now show that both these statements are false for
presymplectic Hamiltonian manifolds.  (However, we will prove in
Appendix \ref{section;normal} that there is a presymplectic
equivariant Darboux theorem under the assumption that the action is
clean.)  Take any isotopic family $(Y_t)_{0\le t\le1}$ of compact
submanifolds of $\R^d$, all of which are transverse to $\phi$.  Then
the manifolds $X_t=\phi^{-1}(Y_t)$ form an isotopic family of compact
submanifolds of $\C^d$, each of which is a presymplectic Hamiltonian
$G$-manifold with presymplectic form $\omega_t=\omega_0|_{X_t}$ and
moment map $\Phi_t=\phi|_{X_t}$.  The forms $\omega_t$ on $X_t$ are
exact for all $t$.  The fibres $X_t$ are equivariantly diffeomorphic,
but they are usually not isomorphic as presymplectic $G$-manifolds.
Indeed, if there existed $G$-equivariant diffeomorphisms $f_t\colon
X_0\to X_t$ satisfying $f_t^*\omega_t=\omega_0$, then we would have
$f_t^*\Phi_t=\Phi_0+\lambda_t$ for some $\lambda_t\in\g^*$, and
therefore $\Phi_t(X_t)=\Phi_0(X_0)+\lambda_t$, which by
\eqref{equation;orthant} would imply
$Y_t\cap\R_{\ge0}=Y_0\cap\R_{\ge0}+\lambda_t$.  So by choosing the
manifolds $Y_t$ in such a way that the intersections
$Y_t\cap\R^d_{\ge0}$ are not translates of each other we can guarantee
that the presymplectic $G$-manifolds $X_t$ are not all isomorphic.

For instance, the $3$-sphere represented by the picture above can be
smoothly deformed to an ellipsoid, the moment map image of which is an
interval.  This deformation is not equivariantly presymplectically
trivial.

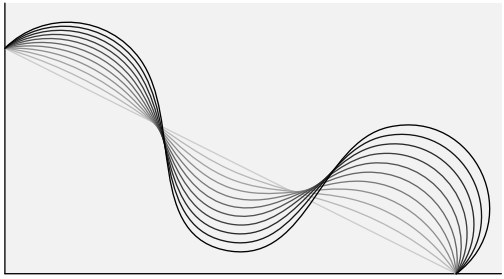
\begin{center}
\includegraphics{nonconvex-2.mps}
\end{center}

Another feature of these examples is that the components of the moment
map $\Phi$ are often not Morse-Bott functions.  For instance, let us
take $Y$ to be a smooth curve in the plane $\R^2$ with the following
properties: (1)~$Y$ is transverse to the coordinate axes;
(2)~$Y\cap\R^2_{\ge0}$ is compact; and (3)~the set $C$ consisting of
all points $y\in Y$ with horizontal tangent line $T_yY$ is countably
infinite and is contained in the open orthant $\R^2_{>0}$.  Then
$X=\phi^{-1}(Y)$ is a compact real hypersurface in $\C^2$.  The
critical set of the second component of the moment map
$\Phi=(\Phi_1,\Phi_2)\colon X\to\R^2$ is equal to
$\Phi^{-1}(C)=\bigcup_{y\in C}\Phi^{-1}(y)$, an infinite disjoint
union of codimension $1$ submanifolds of $X$, and therefore is not a
submanifold.
\end{alinea}

\begin{alinea}[Prato's toric quasifolds]\label{alinea;prato}
We continue the discussion of the previous subsection, but now we take
the submanifold $Y$ in \eqref{equation;orthant} to be an affine
subspace of $\R^d$ transverse to the orthant.  Then the intersection
$P=\Phi(X)=Y\cap\R^d_{\ge0}$ is a convex polyhedron.  The
transversality to the orthant is equivalent to $P$ being
\emph{simple}, i.e.\ the link of each of its faces being a simplex.
It follows from Proposition \ref{proposition;coisotropic} that the
action of $G=\T^d$ on the coisotropic submanifold $X$ of $\C^d$ is
leafwise transitive and that its null ideal $\n(X)$ is the linear
subspace of $\g=\R^d$ orthogonal to $Y$.  Prato
\cite{prato;non-rational-symplectic} calls the leaf space $X/N(X)$ a
\emph{toric quasifold} associated with $P$.  It carries an action of
the quotient group $\T^d/N(X)$ and a moment map whose image is $P$ and
whose fibres are the $\T^d/N(X)$-orbits.  See
\cite{ratiu-zung;presymplectic} for a classification of toric
quasifolds in terms of simple polyhedra.
\end{alinea}

\begin{alinea}[Orbifolds]\label{alinea;orbifolds}
As an illustration of our convexity theorem we present a new proof of
a result due to Lerman et al.\ \cite[Theorem
  1.1]{lerman-meinrenken-tolman-woodward;nonabelian-convexity}.

\begin{subtheorem}\label{theorem;orbifold}
Let $(M,\omega_M)$ be a symplectic orbifold equipped with a
Hamiltonian $G$-action and a proper moment map $\Phi_M$.  Then
$\Delta(M)=\Phi_M(M)\cap C$ is a rational convex polyhedral set.
\end{subtheorem}

\begin{proof}
Let $M_{\textrm{eff}}$ be the effective orbifold which underlies $M$,
as defined e.g.\ in \cite{henriques-metzler;presentations-orbifolds}.
The symplectic structure and the Hamiltonian action descend to
$M_{\textrm{eff}}$, and $M_{\textrm{eff}}$ has the same moment map
image as $M$.  So we may assume without loss of generality that $M$ is
effective.  Choose a $G$-invariant Riemannian metric on $M$ compatible
with the symplectic form $\omega_M$.  This choice endows the tangent
bundle $TM$ with the structure of a Hermitian orbifold vector bundle.
Let $X$ be the unitary frame bundle of $TM$, which is an orbifold
principal bundle over $M$ with structure group $\U(n)$, where
$n=\frac12\dim(M)$.  Then $X$ is a smooth manifold (see
e.g.\ \cite[\S\,2.4]{moerdijk-mrcun;foliations-groupoids}), every
diffeomorphism of $M$ lifts naturally to a $\U(n)$-equivariant
diffeomorphism of $X$, and the $G$-action lifts to a $G$-action on $X$
which commutes with the $\U(n)$-action.  Let $p\colon X\to M$ be the
projection, $\omega=p^*\omega_M$ and $\Phi=p^*\Phi_M$.  Then $X$ is a
presymplectic Hamiltonian $G$-manifold with proper moment map $\Phi$.
The leaves of the null foliation of $X$ are the $\U(n)$-orbits.  The
$G$-action on $X$ is typically not clean (cf.\ Example
\ref{example;clean}), but the action of $\hat{G}=G\times\U(n)$ is
leafwise transitive and has moment map $\hat{\Phi}=\Phi\times0\colon
X\to\hat{\g}^*=\g^*\oplus\lie{u}(n)^*$.  Since the null foliation has
closed leaves, the null subgroup $\hat{N}(X)$ for the $\hat{G}$-action
is a closed subgroup of $\hat{G}$, which contains the subgroup
$\U(n)$.  We conclude from Theorem \ref{theorem;convex} that
$\Delta(M)=\Phi(X)\cap C=\hat{\Phi}(X)\cap C$ is a rational convex
polyhedral set.
\end{proof}

\end{alinea}

\begin{alinea}[Contact manifolds]\label{alinea;contact}
Let $X$ be a compact manifold and let $\alpha$ be a contact $1$-form
on $X$.  The exact $2$-form $\omega=-d\alpha$ is nondegenerate on the
contact hyperplane bundle $\ker(\alpha)$ and therefore is a
presymplectic form of corank $1$.  The null foliation $\ca{F}$ of
$\omega$ is spanned by the Reeb vector field, which is by definition
the unique vector field $\rho$ with the properties
$\iota(\rho)\omega=0$ and $\iota(\rho)\alpha=1$.  Suppose that $G$
acts on $X$ and leaves $\alpha$ invariant.  Then the action is
Hamiltonian with moment map $\Phi\colon X\to\g^*$ defined by
$\Phi^\xi=\iota(\xi_X)\alpha$.  The action is said to be of \emph{Reeb
  type} (see e.g.\ Boyer and Galicki
\cite[\S\,8.4.2]{boyer-galicki;sasakian-geometry}) if there exists a
Lie algebra element $\xi\in\g$ with the property that $\xi_X=\rho$.
Reeb-type actions are leafwise transitive.  The next result follows
immediately from the presymplectic convexity theorem.

\begin{subtheorem}[contact convexity theorem I]
\label{theorem;contact-convex}
Suppose that the $G$-action on $X$ is clean (e.g.\ of Reeb type).
Then $\Delta(X)=\Phi(X)\cap C$ is a convex polytope.
\end{subtheorem}

There are many compact contact Hamiltonian $G$-manifolds $X$ for which
$\Delta(X)$ is \emph{not} convex (and hence the action is not clean).
Here are two methods to produce examples of this.  (1)~Start with a
pair $(X,\alpha)$ for which $\Delta(X)$ is convex and then replace
$\alpha$ by a conformally equivalent contact form $e^f\alpha$ for some
$G$-invariant smooth function $f$.  This has the effect of multiplying
the moment map $\Phi$ by $e^f$, which will usually destroy the
convexity of $\Delta(X)$.  (2)~Follow the method of
\ref{alinea;failure}, starting with a hypersurface $Y$ in $\R^d$ which
is transverse to the faces of the orthant $\R^d_{\ge0}$ as well as to
the radial vector field on $\R^d$.  The moment map $\phi$ for the
standard $\T^d$-action on $\C^d$ is quadratic, and therefore maps the
radial vector field on $\C^d$ to twice the radial vector field on
$\R^d$.  It follows that the hypersurface $X=\phi^{-1}(Y)$ is
transverse to the radial vector field on $\C^d$, which is Liouville,
and hence $X$ is of contact type.  Unless $Y$ is an affine hyperplane
(in which case $X$ is a contact ellipsoid and $\Phi(X)$ a simplex),
the image $\Phi(X)=Y\cap\R^d_{\ge0}$ is not convex.

It is instructive to compare and contrast our contact convexity
theorem, Theorem \ref{theorem;contact-convex}, with a theorem of
Lerman \cite{lerman;convexity-torus-contact}, as improved by Chiang
and Karshon \cite{chiang-karshon;convexity-contact}.  Our result
regards the \emph{symplectic leaf space} $X/\ca{F}$ of $X$.  Their
result regards the \emph{symplectization} $M=X\times(0,\infty)$ of
$X$, which carries the symplectic form
$\omega_M=-d(t\alpha)=t\omega-dt\wedge\alpha$, where $t$ is the
coordinate on the interval $(0,\infty)$.  Letting $G$ act trivially on
the second factor, we get a Hamiltonian $G$-action on $M$ with moment
map $\Phi_M(x,t)=t\Phi(x)$.  The image of $\Phi_M$ is the conical set
$\Phi_M(M)=\bigcup_{t>0}t\Phi(X)$.  Intersecting the image with the
chamber $C$ and adding the origin defines a subset
$\Delta(M)=\{0\}\cup\bigl(\Phi_M(M)\cap C\bigr)$ of $C$, which is the
union of all dilatations of $\Delta(X)$,
$$
\Delta(M)=\bigcup_{t\ge0}t\Delta(X).
$$
The Lerman-Chiang-Karshon theorem states that, if $G$ is a torus of
dimension $\ge2$, then $\Delta(M)$ is a rational convex polyhedral
cone.  They make no cleanness hypothesis on the action.  In fact, by
either of the methods discussed in the previous paragraph one can
manufacture examples where $\Delta(M)$ is convex but $\Delta(X)$ is
not.

Another difference between the two contact convexity theorems is that,
as noted above, the image $\Phi(X)$ is highly dependent on the choice
of the contact form $\alpha$.  In contrast, the symplectic cone
$(M,\omega_M)$ is an intrinsic invariant of the contact hyperplane
bundle $\ker(\alpha)$, from which it follows that its moment cone
$\Delta(M)$ depends only on the conformal class of $\alpha$.

Nevertheless, our Theorem \ref{theorem;contact-convex} is not wholly
independent from the Lerman-Chiang-Karshon theorem.  To see the
connection, note that the Reeb vector field $\rho$ satisfies
$$
\iota(\rho)\omega_M=-t\iota(\rho)d\alpha-\iota(\rho)(dt\wedge\alpha)=
0+dt\wedge\iota(\rho)\alpha=dt,
$$
in other words is the Hamiltonian vector field of the radial function
$t$ on the symplectic cone $M$.  Since $\rho$ is $G$-invariant, this
makes $M$ a Hamiltonian $\hat{G}$-manifold for the product
$\hat{G}=G\times\R$ with moment map $\hat{\Phi}\colon
M\to\g^*\times\R$ defined by $\hat{\Phi}(x,t)=(t\Phi(x),t)$.  Putting
$\hat{\Delta}(M)=\{0\}\cup\bigl(\hat{\Phi}(M)\cap(C\times\R)\bigr)\subset
C\times\R$ we find
\begin{equation}\label{equation;cone}
\hat{\Delta}(M)=\bigcup_{t\ge0}\bigl(t\Delta(X)\times\{t\}\bigr).
\end{equation}
The set $\Delta(X)$ is obtained by intersecting $\hat{\Delta}(M)$ with
the hyperplane $t=1$, which is as it should be, because the symplectic
quotient of $M$ at level $1$ with respect to the $\R$-action is the
leaf space $X/\ca{F}$.  The cone $\Delta(M)$ is obtained by
restricting the $\hat{G}$-action to $G$, i.e.\ by projecting
$\hat{\Delta}(M)$ along the $\R$-axis.  Thus we obtain the following
nonabelian extension of the Lerman-Chiang-Karshon theorem, which
appears to be new.

\begin{subtheorem}[contact convexity theorem II]
Suppose that the $G$-action on $X$ is clean.  Then $\hat{\Delta}(M)$
and $\Delta(M)$ are convex polyhedral cones.
\end{subtheorem}

\begin{proof}
Theorem \ref{theorem;contact-convex} gives that $\Delta(X)$ is a
convex polytope.  Let $\eta_1$, $\eta_2,\dots$, $\eta_k\in\t$ be
inward-pointing normal vectors to the facets of $\Delta(X)$; then
points $v\in\Delta(X)$ are determined by inequalities of the form
$\inner{\eta_i,v}\ge a_i$ with $a_i\in\R$.  It then follows from
\eqref{equation;cone} that points $\hat{v}=(tv,t)\in \hat{\Delta}(M)$
are determined by the homogeneous inequalities
$\inner{\eta_i,tv}-a_it\ge0$,
i.e.\ $\inner{\hat{\eta}_i,\hat{v}}\ge0$, where
$\hat{\eta}_i=(\eta_i,-a_i)\in\t\times\R$.  Hence $\hat{\Delta}(M)$ is
a convex polyhedral cone.  Hence its projection $\Delta(M)$ onto
$\t^*$ is likewise a convex polyhedral cone.
\end{proof}

Unfortunately our proof does not enable us to show that the cone
$\Delta(M)$ is rational, nor that it is convex if the action is not
clean.
\end{alinea}

%%%%%%%%%%%%%%%%%%%%%%%%%%%%%%%%%%%%%%%%%%%%%%%%%%%%%%%%%%%%%%%%%%%%%%%%
\appendix
%%%%%%%%%%%%%%%%%%%%%%%%%%%%%%%%%%%%%%%%%%%%%%%%%%%%%%%%%%%%%%%%%%%%%%%%

%%%%%%%%%%%%%%%%%%%%%%%%%%%%%%%%%%%%%%%%%%%%%%%%%%%%%%%%%%%%%%%%%%%%%%%%
\section{Immersed normal subgroups}\label{section;immersed}
%%%%%%%%%%%%%%%%%%%%%%%%%%%%%%%%%%%%%%%%%%%%%%%%%%%%%%%%%%%%%%%%%%%%%%%%

Let $G$ be a connected compact Lie group.  What immersed connected
normal Lie subgroups $N$ does $G$ have?  ``Not very many'' is the
answer.  There are two basic types of such immersions: (1)~$G$ is
semisimple and simply connected and $N$ is a product of simple factors
of $G$; (2)~$G$ is a torus and $N$ is a product of a torus and a
vector space immersing into $G$.  Type (1) is a closed embedding, but
type (2) may not be.  We have the following straightforward result.

\begin{lemma}
Every immersed connected normal Lie subgroup $N$ of $G$ is, up to
finite covering groups, a product of types \emph{(1)} and \emph{(2)}.
\end{lemma}

\begin{proof}
The Lie algebra $\g=\Lie(G)$ is the direct sum of the derived
subalgebra $\g_1=[\g,\g]$ and the centre $\g_2=\lie{z}(\g)$.  The
ideal $\n=\Lie(N)$ is the direct sum of the ideals $\n_1=\n\cap\g_1$
and $\n_2=\n\cap\g_2$.  The ideal $\n_1$ is a direct sum of simple
ideals of the semisimple Lie algebra $\g_1$.  (See
e.g.\ \cite[\S~I.6]{bourbaki;groupes-algebres}.)  Letting $G_1$ and
$N_1$ be the corresponding simply connected groups, we have an
embedding $N_1\to G_1$ of type (1).  Letting $G_2$ be the identity
component of the centre $Z(G)$ and $N_2=\exp(\n_2)$, we have an
immersion $N_2\to G_2$ of type (2).  The product $G_1\times G_2$ is a
finite covering group of $G$ and $N_1\times N_2$ is a finite covering
group of the identity component of $N$.
\end{proof}

The following consequence is used in the proof of Theorem
\ref{theorem;convex}\eqref{item;rational}.  As usual, by a rational
subspace of the Lie algebra of a torus we mean a subspace that is
rational with respect to the $\Q$-structure defined by the character
lattice of the torus.

\begin{corollary}\label{corollary;closed-rational}
Let $N$ be an immersed normal Lie subgroup of $G$ and $\t$ a Cartan
subalgebra of $\g$.  The following statements are equivalent:
\begin{enumerate}
\item $N$ is closed;
\item $\n\cap\lie{z}(\g)$ is a rational subspace of $\lie{z}(\g)$;
\item $\n\cap\t$ is a rational subspace of $\t$.
\end{enumerate}
\end{corollary}  

%%%%%%%%%%%%%%%%%%%%%%%%%%%%%%%%%%%%%%%%%%%%%%%%%%%%%%%%%%%%%%%%%%%%%%%%
\section{Some presymplectic linear algebra}\label{section;presymplectic}
%%%%%%%%%%%%%%%%%%%%%%%%%%%%%%%%%%%%%%%%%%%%%%%%%%%%%%%%%%%%%%%%%%%%%%%%

This appendix lists a few elementary facts referred to in the proof of
the convexity theorem.  Let $G$ be a fixed connected compact Lie
group.  Let $E$ be a finite-dimensional real vector space and $F$ a
linear subspace.  We denote by $F^\circ\subset E^*$ the annihilator of
$F$.  Let $\sigma$ be a presymplectic form on $E$.  We denote by
$$
F^\sigma=\{\,u\in E\mid\text{$\sigma(u,v)=0$ for all $v\in F$}\,\}
$$
the subspace of $E$ orthogonal to $F$ with respect to $\sigma$.  We
call $E$ a \emph{presymplectic $G$-module} if $G$ acts smoothly and
linearly on $E$ and the $G$-action preserves $\sigma$.  We state the
following simple result without proof.

\begin{lemma}\label{lemma;symplectic-module}
A presymplectic $G$-module $(E,\sigma)$ is a presymplectic Hamiltonian
$G$-manifold with moment map $\Phi_E$ given by
$\Phi_E^\xi(e)=\frac12\sigma(\xi(e),e)$, where $e\mapsto\xi(e)$
denotes the action of $\xi\in\g$ on $E$.
\end{lemma}  

We require the following simple lemma concerning linear $G$-actions.

\begin{lemma}\label{lemma;module-affine}
Let $E$ be finite-dimensional real $G$-module.  Let $F$ be a
$G$-submodule and $e\in E$.  Then the following conditions are
equivalent.
\begin{enumerate}
\item\label{item;decompose-fixed}
$e=e_0+e_1$ for some $e_0\in E^G$ and some $e_1\in F$;
\item\label{item;orbit-affine}
$G\cdot e\subset e+F$;
\item\label{item;affine-invariant}
the affine subspace $e+F$ is preserved by the $G$-action;
\item\label{item;tangent-affine}
$T_e(G\cdot e)\subset F$.
\end{enumerate}
\end{lemma}  

\begin{proof}
First we prove \eqref{item;orbit-affine} $\implies$
\eqref{item;decompose-fixed} (which is the only implication that
requires the compactness of $G$).  Let $dg$ be the normalized Haar
measure on $G$ and put $e_0=\int_Gg\cdot e\,dg$.  Then $e_0$ is fixed
under the action.  Since $g\cdot e\in e+F$ for all $g\in G$, we have
$\phi(g\cdot e)=\phi(e)$ for all $\phi\in F^\circ$.  Therefore
$$
\phi(e_0)=\int_G\phi(g\cdot e)\,dg= \int_G\phi(e)\,dg=\phi(e)
$$
for all $\phi\in F^\circ$.  This shows that $e-e_0\in F$, which proves
\eqref{item;decompose-fixed}.  Next we prove
\eqref{item;tangent-affine} $\implies$ \eqref{item;orbit-affine}
(which is the only implication that requires the connectedness of
$G$).  If $T_e(G\cdot e)\subset F$, then $T_f(G\cdot e)\subset F$ for
all $f\in G\cdot e$, because the action preserves $F$.  Hence the
orbit $G\cdot e$ is everywhere tangent to the foliation of $E$ given
by the affine subspaces parallel to $F$.  Therefore $G\cdot e$ is
contained in the leaf $e+F$.  The other implications are
straightforward.
\end{proof}

\begin{lemma}\label{lemma;coadjoint}
Let $\a$ be an ideal of $\g$ and $\lambda\in\g^*$.  The coadjoint
orbit $G\cdot\lambda$ is contained in the affine subspace
$\lambda+\a^\circ$ if and only if $\a$ is contained in the centralizer
$\g_\lambda$ of $\lambda$.
\end{lemma}

\begin{proof}
Apply Lemma \ref{lemma;module-affine} to the module $E=\g^*$ and the
submodule $F=\a^\circ$ and use the fact that
$T_\lambda(G\cdot\lambda)=\g_\lambda^\circ$.
\end{proof}

In the next statements $X$ denotes a presymplectic Hamiltonian
$G$-manifold with moment map $\Phi$ and null foliation $\ca{F}$.  The
result holds regardless of any cleanness assumptions on the $G$-action
on $X$.  We denote by $\bar{x}$ the leaf of a point $x\in X$,
considered as a point in the leaf space $X/\ca{F}$, and by
$\g_{\bar{x}}$ its stabilizer subalgebra, which consists of all
$\xi\in\g$ satisfying $\xi_X(x)\in T_x\ca{F}$.

\begin{lemma}\label{lemma;image-kernel}
For all $x\in X$ we have
\begin{enumerate}
\item\label{item;kernel}
$\ker(T_x\Phi)=T_x(G\cdot x)^{\omega_x}$;
\item\label{item;image}
$\im(T_x\Phi)=\g_{\bar{x}}^\circ$ and
  $\coker(T_x\Phi)=\g_{\bar{x}}^*$.
\end{enumerate}
\end{lemma}

\begin{proof}
\eqref{item;kernel}~The moment map condition
$d\Phi^\xi=\iota(\xi_X)\omega$ implies that $v\in\ker(T_x\Phi)$ if and
only if $\omega_x(\xi_X(x),v)=0$, i.e.\ $v\in T_x(G\cdot
x)^{\omega_x}$.

\eqref{item;image}~Similarly, a vector $\xi\in\g$ is in
$\im(T_x\Phi)^\circ$ if and only if $\omega_x(\xi_X(x),v)=0$ for all
$v\in T_xX$, i.e.\ $\xi_X(x)\in\ker(\omega_x)=T_x\ca{F}$.  This is
equivalent to $\xi_X(x)\in T_x(G\cdot x)\cap
T_x\ca{F}=T_x(G_{\bar{x}}\cdot x)$, where we used
\eqref{equation;stabilizer}.  Thus $\im(T_x\Phi)^\circ=\g_{\bar{x}}$.
In other words, $\im(T_x\Phi)=\g_{\bar{x}}^\circ$ and
$\coker(T_x\Phi)=\g^*/\g_{\bar{x}}^\circ=\g_{\bar{x}}^*$.
\end{proof}

Recall that any smooth map of manifolds $f\colon A\to B$ has an
intrinsically defined Hessian or second derivative
$T^2_af\colon\ker(T_af)\to\coker(T_af)$ at every $a\in A$.

\begin{corollary}\label{corollary;hessian}
For every $x\in X$ the second derivative of the moment map at $x$ is a
linear map $T^2_x\Phi\colon T_x(G\cdot
x)^{\omega_x}\to\g_{\bar{x}}^*$.  In particular, if the leaf $\bar{x}$
is $G$-fixed, the second derivative is a linear map $T^2_x\Phi\colon
T_xX\to\g^*$.
\end{corollary}

%%%%%%%%%%%%%%%%%%%%%%%%%%%%%%%%%%%%%%%%%%%%%%%%%%%%%%%%%%%%%%%%%%%%%%%%
\section{The local normal form}\label{section;normal}
%%%%%%%%%%%%%%%%%%%%%%%%%%%%%%%%%%%%%%%%%%%%%%%%%%%%%%%%%%%%%%%%%%%%%%%%

\begin{alinea}
This appendix, the results of which are not used in the main body of
the paper, but which develops a theme touched upon in Section
\ref{alinea;failure}, contains a local normal form theorem for clean
presymplectic Hamiltonian Lie group actions, which is a refinement of
the slice theorem, Theorem \ref{theorem;transversal}.  It is a type of
equivariant Darboux-Weinstein theorem, which extends results
established in the symplectic case by Guillemin and Sternberg
\cite[\S\,41]{guillemin-sternberg;techniques;;1990} and Marle
\cite{marle;modele-action}.  It says that up to isomorphism an
invariant neighbourhood of a point $x$ in a presymplectic Hamiltonian
$G$-manifold is entirely determined by infinitesimal data, namely the
stabilizer subgroup $G_x$, the image of $x$ under the moment map, and
two finite-dimensional representations of $G_x$, which describe the
relevant information ``orthogonal'' to the orbit of $x$.  These
$G_x$-modules, which are subquotients of the tangent space at $x$ and
which we call the \emph{symplectic slice} and the \emph{null slice},
capture respectively the symplectic directions and the null directions
complementary to the orbit.  As shown in Section \ref{alinea;failure},
the equivariant Darboux-Weinstein theorem is false without a cleanness
assumption on the point~$x$.  For a fixed point $x$ the theorem
reduces to Corollary \ref{corollary;linear}.

As in Section \ref{section;convex} we denote by $X$ a manifold with
presymplectic form $\omega$ and by $G$ a connected compact Lie group
acting on $M$ in a Hamiltonian fashion with moment map $\Phi$.  (See
\ref{alinea;presymplectic}.)  We denote the null foliation of $\omega$
by $\ca{F}=\ca{F}_X$, the null ideal sheaf by
$\tilde{\n}=\tilde{\n}_X$ and the null subgroup sheaf by
$\tilde{N}=\tilde{N}_X$.  (See \ref{alinea;null}.)  In
\ref{alinea;reduction}--\ref{alinea;slice} we introduce the slice
modules; in \ref{alinea;model} we describe the local model and in
\ref{alinea;darboux} we prove the equivariant Darboux-Weinstein
theorem.
\end{alinea}

\begin{alinea}\label{alinea;reduction}
Let $E$ be a finite-dimensional real vector space and $\sigma$ a
presymplectic form on $E$.  We denote by $E^\natural$ the largest
symplectic quotient space of $E$, i.e.\ $E^\natural=E/\ker(\sigma)$,
and by $\sigma^\natural$ the symplectic form on $E^\natural$ induced
by $\sigma$.  If $F$ is a linear subspace of $E$ equipped with the
presymplectic form $\sigma_F=\sigma|_F$, then $\ker(\sigma_F)=F\cap
F^\sigma$, where $F^\sigma$ denotes the subspace of $E$ orthogonal to
$F$ with respect to $\sigma$.  Therefore the largest symplectic
quotient space of $F$ is
$$F^\natural=F/(F\cap F^\sigma)\cong(F+F^\sigma)/F^\sigma.$$
We have $(F^\sigma)^\sigma=F+\ker(\sigma)$ and so
$$
\ker(\sigma_{F^\sigma})=F^\sigma\cap(F+\ker(\sigma))=F^\sigma\cap
F+\ker(\sigma).
$$
Hence the largest symplectic quotient space of $F^\sigma$ is
\begin{equation}\label{equation;reduction}
(F^\sigma)^\natural=F^\sigma/(F^\sigma\cap
F+\ker(\sigma))\cong(F+F^\sigma)/(F+\ker(\sigma)).
\end{equation}
Let $H$ be a Lie group and suppose that $E$ is a presymplectic
$H$-module.  Suppose also that $F$ is an $H$-submodule of $E$.  Then
the vector spaces $E^\natural$, $F^\natural$ and $(F^\sigma)^\natural$
are symplectic $H$-modules in a natural way.  We omit the proof of the
following elementary assertion.

\begin{sublemma}\label{lemma;natural}
Let $(E_1,\sigma_1)$ and $(E_2,\sigma_2)$ be presymplectic $H$-modules.
\begin{enumerate}
\item\label{item;surjective}
Let $f\colon E_1\to E_2$ be an $H$-equivariant surjective linear map
satisfying $f^*\sigma_2=\sigma_1$.  Then $f$ descends to an
isomorphism of symplectic $H$-modules
$$
f^\natural\colon E_1^\natural\overset\cong\longto E_2^\natural.
$$
Let $F_1$ be an $H$-submodule of $E_1$ and put $F_2=f(F_1)$.  Then
$F_2^{\sigma_2}=f(F_1^{\sigma_1})$.  Hence $f$ descends to
isomorphisms of symplectic $H$-modules
$$
f^\natural\colon F_1^\natural\overset\cong\longto F_2^\natural,\qquad
f^\natural\colon\bigl(F_1^{\sigma_1}\bigr)^\natural\overset\cong\longto
\bigl(F_2^{\sigma_2}\bigr)^\natural.
$$
\item\label{item;sum}
Let $F_1$ be an $H$-submodule of $E_1$ and $F_2$ an $H$-submodule of
$E_2$.  Let $E=E_1\oplus E_2$ and $F=F_1\oplus F_2$.  Then
$E^\natural\cong E_1^\natural\oplus E_2^\natural$, $F^\natural\cong
F_1^\natural\oplus F_2^\natural$, and $(F^\sigma)^\natural\cong
(F_1^{\sigma_1})^\natural\oplus(F_2^{\sigma_2})^\natural$.
\end{enumerate}  
\end{sublemma}

\end{alinea}

\begin{alinea}\label{alinea;slice}
Let $x\in X$.  Applying the observations of \ref{alinea;reduction} to
the Lie group $H=G_x$, the $G_x$-module $E=T_xX$, the presymplectic
form $\sigma=\omega_x$, and the submodule $F=T_x(G\cdot x)$, we arrive
at a symplectic $G_x$-module
$$
S_x(X)=\bigl(T_x(G\cdot x)^{\omega_x}\bigr)^\natural=T_x(G\cdot
x)^{\omega_x}\big/\bigl(T_x(G\cdot x)^{\omega_x}\cap T_x(G\cdot
x)+T_x\ca{F})\bigr),
$$
which we call the \emph{symplectic slice} of $X$ at $x$.

\begin{sublemma}\label{lemma;slice-ideal}
The symplectic slice $S_x(X)$ is naturally isomorphic to a submodule
of the module $V_1=T_xX\big/\bigl(T_x(G\cdot x)+T_x\ca{F}\bigr)$ of
Theorem \ref{theorem;transversal}.  The ideal $\g_x\cap\n_x$ of $\g_x$
acts trivially on $S_x(X)$.
\end{sublemma}

\begin{proof}
It follows from \eqref{equation;reduction} that
$$
S_x(X)\cong\bigl(T_x(G\cdot x)+T_x(G\cdot
x)^{\omega_x}\bigr)\big/\bigl(T_x(G\cdot x)+T_x\ca{F}\bigr),
$$
which is a submodule of $V_1$.  Let $\eta\in\g_x\cap\n_x$.  Then the
function $\Phi^\eta$ is constant near $x$, and therefore its Hessian
$T_x^2\Phi^\eta\colon T_xX\to\R$ is $0$.  By the equivariant Darboux
theorem, Corollary \ref{corollary;linear}, applied to the fixed point
$x$ of the $G_x$-action, the Hessian $T_x^2\Phi^\eta$ is the
$\eta$-component of the moment map of the linear $G_x$-action on
$T_xX$.  Therefore the linearization at $x$ of the vector field
$\eta_X$ is tangent to the leaves of the constant presymplectic form
$\omega_x$ on $T_xX$.  It follows that $\eta$ acts trivially on the
quotient module $(T_xX)^\natural=T_xX/T_x\ca{F}$.  Hence $\eta$ acts
trivially on the subquotient $S_x(X)$ of $(T_xX)^\natural$.
\end{proof}

Let $\bar{x}=\ca{F}(x)$ be the leaf of $x$, considered as a point in
the leaf space $X/\ca{F}$, and let $G_{\bar{x}}$ the stabilizer of
$\bar{x}$.  The \emph{null slice} of $X$ at $x$ is the $G_x$-module
$$
V_x(X)=\bigl(T_x(G\cdot x)+T_x\ca{F}\bigr)\big/T_x(G\cdot x)\cong
T_x\ca{F}\big/\bigl(T_x(G\cdot x)\cap T_x\ca{F}\bigr)=
T_x\ca{F}/T_x(G_{\bar{x}}\cdot x),
$$
where the last equality follows from \eqref{equation;stabilizer}.
This is the module denoted by $V_0$ in Theorem
\ref{theorem;transversal}.  Note that $V_x(X)=0$ if and only if the
leaf $\ca{F}(x)$ is contained in the $G$-orbit of $x$.  The next
result describes how the symplectic slice and the null slice behave
under presymplectic submersions and under symplectization.

\begin{sublemma}\label{lemma;slice}
Let $x\in X$.
\begin{enumerate}
\item\label{item;transverse-slice}
Let $(Y,\omega_Y)$ be a presymplectic Hamiltonian $G$-manifold and let
$p\colon X\to Y$ be an equivariant submersion with the property
$p^*\omega_Y=\omega$.  Let $y=p(x)$ and assume $G_y=G_x$.  Then $p$
induces an isomorphism of symplectic $G_x$-modules
$$p^\natural\colon S_x(X)\overset\cong\longto S_x(Y)$$
and a short exact sequence of $G_x$-modules
$$
0\longto\ker(T_xp)\longto V_x(X)\overset{T_xp}\longto
V_y(Y)\longto0.
$$
\item\label{item;symplectization-slice}
Let $M=T^*\ca{F}$ be the symplectization of $X$.  Then $V_x(M)=0$ and
there is an isomorphism of symplectic $G_x$-modules 
$$S_x(M)\cong S_x(X)\oplus V_x(X)\oplus V_x(X)^*.$$
In particular, $S_x(M)\cong S_x(X)$ if $\ca{F}(x)\subset G\cdot x$.
\end{enumerate}
\end{sublemma}  

\begin{proof}
\eqref{item;transverse-slice}~Let $H=G_x=G_y$, $E_1=T_xX$, and
$E_2=T_yY$.  Then $E_1$ and $E_2$ are presymplectic $H$-modules with
presymplectic forms $\sigma_1=\omega_x$,
resp.\ $\sigma_2=\omega_{Y,y}$.  The tangent spaces to the orbits
$F_1=T_x(G\cdot x)$ and $F_2=T_y(G\cdot y)$ are submodules of $E_1$,
resp.\ $E_2$.  The tangent map $p_*=T_xp\colon E_1\to E_2$ satisfies
$p^*\sigma_2=\sigma_1$ and $p_*(F_1)=F_2$.  By definition the
symplectic slices are $S_x(X)=(F^{\sigma_1})^\natural$ and
$S_y(Y)=(F^{\sigma_2})^\natural$.  The statement that the two are
isomorphic now follows from the third isomorphism in Lemma
\ref{lemma;natural}\eqref{item;surjective}.  The restriction of $p_*$
to the subspace $F_1+T_x\ca{F}$ has image $F_2+T_y\ca{F}_Y$.  Hence
$p_*$ descends to a surjection from $V_x(X)=(F_1+T_x\ca{F})/F_1$ to
$V_y(Y)=(F_2+T_y\ca{F}_Y)/F_2$ with kernel $\ker(p_*)$.

\eqref{item;symplectization-slice}~Since $M$ is symplectic, we have
$V_x(M)=0$.  Let $E_1=T_xX$ and $E_2=E_1^\natural$.  On $E_1$ we have
the presymplectic form $\sigma_1=\omega_x$ and on $E_2$ we have the
symplectic form $\sigma_2=\sigma_1^\natural$.  Let $\pi\colon E_1\to
E_2$ be the quotient map and $E_0=\ker(\pi)$ its kernel.  Then
$E_0=\ker(\sigma_1)$ and $\pi^*\sigma_2=\sigma_1$.  Let
$F_1=T_x(G\cdot x)\cong\g/\g_x$, let $F_2=\pi(F_1)\subset E_2$ be the
image of $F_1$, and let $F_0=F_1\cap E_0$ be the kernel of $\pi\colon
F_1\to F_2$.  It follows from \eqref{equation;stabilizer} that
$$
F_0=T_x(G\cdot x)\cap T_x\ca{F}=T_x(G_{\bar{x}}\cdot
x)\cong\g_{\bar{x}}/\g_x.
$$  
Therefore $F_2=F_1/F_0\cong\g/\g_{\bar{x}}$.  From the third
isomorphism in Lemma \ref{lemma;natural}\eqref{item;surjective} we
obtain
\begin{equation}\label{equation;slice}
S_x(X)=\bigl(F_1^{\sigma_1}\bigr)^\natural\cong
\bigl(F_2^{\sigma_2}\bigr)^\natural=
F_2^{\sigma_2}\big/\bigl(F_2^{\sigma_2}\cap F_2\bigr).
\end{equation}
We can express the relationships among the various presymplectic
$G_x$-modules as a commutative diagram with exact rows:
\begin{equation}\label{equation;slices}
\vcenter{\xymatrix@H=1.5em{
0\ar[r]&\g_{\bar{x}}/\g_x\ar[r]\ar[d]_{\cong}&\g/\g_x\ar[r]\ar[d]_{\cong}&
\g/\g_{\bar{x}}\ar[r]\ar[d]_{\cong}&0
\\
0\ar[r]&F_0\ar[r]\ar@{^{(}->}[d]&F_1\ar[r]\ar@{^{(}->}[d]&
F_2\ar[r]\ar@{^{(}->}[d]&0
\\
0\ar[r]&E_0\ar[r]&E_1\ar[r]&E_2\ar[r]&0
}}
\end{equation}
Recall from \ref{alinea;coisotropic} that the symplectic form $\Omega$
on $M=T^*\ca{F}$ depends on a choice of a $G$-invariant Riemannian
metric on $X$.  Given such a choice, we obtain compatible splittings
of the rows of the diagram \eqref{equation;slices}, namely by
identifying the middle terms with orthogonal direct sums $E_1\cong
E_2\oplus E_0$ and $F_1\cong F_2\oplus F_0$.  Let $E=T_xM$ and
$\sigma=\Omega_x$.  Then $E$ is a direct sum, $E=E_2\oplus E_0\oplus
E_0^*$, and the symplectic form on $E$ is
$\sigma=\sigma_2\oplus\sigma_0$, where $\sigma_0$ is the standard
symplectic form on $E_0\oplus E_0^*$.  Applying Lemma
\ref{lemma;natural}\eqref{item;sum} to the subspace $F_1=F_2\oplus
F_0$ of $E$ we obtain
\begin{equation}\label{equation;bigslice}
S_x(M)=F_2^{\sigma_2}\big/\bigl(F_2^{\sigma_2}\cap F_2\bigr)\oplus
F_0^{\sigma_0}\big/\bigl(F_0^{\sigma_0}\cap F_0\bigr).
\end{equation}
Now $F_0^{\sigma_0}=E_0\oplus F_0^\circ$, where $F_0^\circ$ is the
annihilator of $F_0$ in $E_0^*$, so $F_0^{\sigma_0}\cap F_0=F_0$ and
$$
F_0^{\sigma_0}\big/\bigl(F_0^{\sigma_0}\cap F_0\bigr)=E_0/F_0\oplus
F_0^\circ=V\oplus V^*,
$$
where $V=V_x(X)$.  Substituting this and \eqref{equation;slice} into
\eqref{equation;bigslice} gives $S_x(M)\cong S_x(X)\oplus V\oplus
V^*$.  If $\ca{F}(x)\subset G\cdot x$, then $V=0$, so $S_x(M)\cong
S_x(X)$.
\end{proof}

\end{alinea}  

\begin{alinea}\label{alinea;model}
We now describe the local model for clean presymplectic Hamiltonian
actions.  First a quick review of the symplectic case.  (See
\cite[\S\,41]{guillemin-sternberg;techniques;;1990} or
\cite{marle;modele-action} for a complete exposition.)  The symplectic
local model has a list of four ingredients $(\lambda,H,\theta,S)$,
consisting of
\begin{enumerate}
\renewcommand\theenumi{\arabic{enumi}}
\item\label{item;covector}
a covector $\lambda\in\g^*$,
\item\label{item;subgroup}
a closed subgroup $H$ of the coadjoint stabilizer $G_\lambda$ of
$\lambda$,
\item\label{item;splitting}
an $H$-equivariant splitting
$\theta\colon\g_\lambda/\h\to\g_\lambda$ of the quotient map
$\g_\lambda\to\g_\lambda/\h$,
\item\label{item;slice}
a finite-dimensional symplectic $H$-module $S$.
\end{enumerate}
We denote the $H$-module $\g_\lambda/\h$ by $\m$.  We use the
splitting $\theta$ to identify the $H$-module $\m$ with a direct
summand of $\g_\lambda$ and the $H$-module $\h^*$ with a direct
summand of $\g_\lambda^*$.  The homogeneous bundle
\begin{equation}\label{equation;symplectic-model}
\lie{M}=\lie{M}(\lambda,H,\theta,S)=G\times^H(\m^*\times S)
\end{equation}
carries a closed $2$-form $\omega_{\lie{M}}$ which is nondegenerate in
a neighbourhood of the zero section.  The left multiplication action
of $G$ on $\lie{M}$ is Hamiltonian with moment map $\Phi_{\lie{M}}$
given by
$$\Phi_{\lie{M}}([g,a,s])=\Ad_g^*\bigl(\lambda+a+\Phi_S(s)\bigr),$$
where $\Phi_S^\eta(s)=\frac12\omega_S(\eta_S(s),s)$ for $\eta\in\h$.
The formula for $\Phi_{\lie{M}}$ is to be interpreted as follows.  The
inclusion $\g_\lambda\to\g$ has a unique $G_\lambda$-equivariant left
inverse, which we use to identify $\g_\lambda^*$ with a direct summand
of $\g^*$.  This allows us to identify $\g^*$ with a product
\begin{equation}\label{equation;gstar}
\g^*\cong\g_\lambda^\circ\times\g_\lambda^*\cong
\g_\lambda^\circ\times\m^*\times\h^*,
\end{equation}
and to regard $a\in\m^*$ and $\Phi_S(s)\in\h^*$ as elements of $\g^*$.
Then for $g\in G$ we let $\Ad_g^*$ act on the element
$\lambda+a+\Phi_S(s)\in\g^*$.

The presymplectic local model requires six ingredients
$(\lambda,H,\theta,S,V,\a)$, where $\lambda$, $H$, $\theta$, $S$ are
as in \eqref{item;covector}--\eqref{item;slice} and in addition we
have
\begin{enumerate}
\renewcommand\theenumi{\arabic{enumi}}
\addtocounter{enumi}{4}
\item\label{item;module}
an $H$-module $V$,
\item\label{item;ideal}
an ideal $\a$ of $\g$ with the properties that
$\lie{k}\subset\a\subset\g_\lambda$ and that the ideal $\a\cap\h$ of
$\h$ acts trivially on $S$.
\end{enumerate}
Here $\lie{k}$ denotes the kernel of the infinitesimal $G$-action
$\g\to\Gamma(T\lie{M})$ on the symplectic Hamiltonian $G$-manifold
$\lie{M}$ defined in \eqref{equation;symplectic-model}.  The ideal
$\lie{k}$ of $\g$ is determined by the data $\lambda$, $H$ and $S$.
The quotient $\p=(\a+\h)/\h\cong\a/(\a\cap\h)$ is an $H$-submodule of
$\m=\g_\lambda/\h$.  We require the splitting
$\theta\colon\m\to\g_\lambda$ to be \emph{compatible} with the ideal
$\a$ in the sense that $\theta(\p)$ is contained in $\a$.  Because the
$H$-action on $\g_\lambda$ preserves the ideal $\a$, such a splitting
always exists.  We form the quotient module $\q=\m/\p$ and the
homogeneous vector bundle
$$
\lie{X}=\lie{X}(\lambda,H,\theta,S,V,\a)=G\times^H(\q^*\times S\times
V).
$$
We define an equivariant vector bundle map $f\colon\lie{X}\to\lie{M}$
by
$$f([g,b,s,v])=[g,i(b),s],$$
where $i\colon\q^*\to\m^*$ is the natural inclusion.  Let $\lie{Y}$ be
the direct summand
$$\lie{Y}=\lie{X}(\lambda,H,\theta,S,0,\a)=G\times^H(\q^*\times S)$$
of the vector bundle $\lie{X}$.  Then $f=j\circ p$, where
$$
\lie{X}\overset{p}\longto\lie{Y}\overset{j}\longto\lie{M}
$$
are defined by $p([g,b,s,v])=[g,b,s]$ and $j([g,b,s])=[g,i(b),s]$.  We
write $\omega_{\lie{X}}=f^*\omega_{\lie{M}}$,
$\Phi_{\lie{X}}=f^*\Phi_{\lie{M}}$,
$\omega_{\lie{Y}}=j^*\omega_{\lie{M}}$, and
$\Phi_{\lie{Y}}=j^*\Phi_{\lie{M}}$.  Note that $j$ is an equivariant
embedding of the vector bundle $\lie{Y}$.  We identify $\lie{Y}$ with
the subbundle $j(\lie{Y})$ of $\lie{M}$.  Let $x_0\in\lie{X}$ denote
the basepoint $[1,0,0,0]$ and let $A$ be the connected immersed normal
subgroup of $G$ generated by the ideal $\a$.  Here are the relevant
properties of the model $\lie{X}$.

\begin{sublemma}\label{lemma;model}
\begin{enumerate}
\item\label{item;model-coisotropic}
$\Phi_{\lie{M}}\colon\lie{M}\to\g^*$ intersects the affine subspace
  $\lambda+\a^\circ$ cleanly and
  $\lie{Y}=\Phi_{\lie{M}}^{-1}(\lambda+\a^\circ)$.  Hence near the
  zero section $\lie{Y}$ is a coisotropic submanifold of $\lie{M}$ and
  $\lie{M}$ is the symplectization of $\lie{Y}$.  The leaves of the
  null foliation of $\omega_{\lie{Y}}$ are the orbits of the
  $A$-action on $\lie{Y}$.
\item\label{item;model-presymplectic}
Near the zero section $\lie{X}$ is a presymplectic Hamiltonian
$G$-manifold with presymplectic form $\omega_{\lie{X}}$ and moment map
$\Phi_{\lie{X}}$.  The $G$-action on $\lie{X}$ is clean at $x_0$.
Near the zero section the null ideal sheaf $\tilde{\n}_{\lie{X}}$ is
constant with stalk $\a$.
\item\label{item;model-stabilizer}
The stabilizer of the basepoint is $G_{x_0}=H$, its moment map value
is $\Phi_{\lie{X}}(x_0)=\lambda$, the symplectic slice is
$S_{x_0}(\lie{X})\cong S$, and the null slice is
$V_{x_0}(\lie{X})\cong V$.
\end{enumerate}  
\end{sublemma}

\begin{proof}
\eqref{item;model-coisotropic}~Let $[g,a,s]\in\lie{M}$.  Writing
$\Phi_{\lie{M}}([g,a,s])=\Ad_g^*(\lambda+\phi(a,s))$, where $\phi$ is
the $H$-equivariant map $\m^*\times S\to\g^*$ defined by $(a,s)\mapsto
a+\Phi_S(s)$, we have
$$
\Phi_{\lie{M}}([g,a,s])\in\lambda+\a^\circ\iff\phi(a,s)\in\a^\circ.
$$
Under the identification \eqref{equation;gstar} we have
$\a^\circ\cong\g_\lambda^\circ\times\q^*\times(\a\cap\h)^\circ$, where
$\g_\lambda^\circ$ denotes the annihilator of $\g_\lambda$ in $\g^*$
and $(\a\cap\h)^\circ$ the annihilator of $\a\cap\h$ in $\h^*$.  By
assumption $\a\cap\h$ acts trivially on $S$, so the $H$-moment map
$\Phi_S$ maps $S$ into $(\a\cap\h)^\circ$.  Therefore
$\phi(a,s)\in\a^\circ$ is equivalent to $a\in\q^*$,
i.e.\ $\phi^{-1}(\a^\circ)=\q^*\times S$.  Therefore
$\lie{Y}=G\cdot\phi^{-1}(\a^\circ)=\Phi_{\lie{M}}^{-1}(\lambda+\a^\circ)$.
Moreover, $\phi$ intersects the linear subspace $\a^\circ$ cleanly,
which implies that $\Phi_{\lie{M}}$ intersects $\lambda+\a^\circ$
cleanly.  The remaining assertions now follow from Proposition
\ref{proposition;coisotropic}.

\eqref{item;model-presymplectic}~The first assertion follows from
\eqref{item;model-coisotropic} and the fact that
$p\colon\lie{X}\to\lie{Y}$ is an equivariant surjective submersion.
By Proposition \ref{proposition;coisotropic} the action on $\lie{Y}$
is leafwise transitive and the null ideal is $\n(\lie{Y})=\a$.
Therefore the action on $\lie{X}$ is clean at $x_0$ and, by Corollary
\ref{corollary;clean}, the sheaf $\n_{\lie{X}}$ is constant near $x_0$
with stalk $\a$.

\eqref{item;model-stabilizer}~The space $\lie{X}$ is a homogeneous
bundle over $G/H$ and the basepoint $x_0$ is the identity coset in
$G/H$, so its stabilizer is $G_{x_0}=H$.  We have
$\Phi_{\lie{X}}(x_0)=\Phi_{\lie{M}}([1,0,0])=\lambda$.  The
equivariant surjection $p$ induces an isomorphism
$S_{x_0}(\lie{X})\cong S_{x_0}(\lie{Y})$ by Lemma
\ref{lemma;slice}\eqref{item;transverse-slice}.  The fact that the
action on $\lie{Y}$ is leafwise transitive implies
$V_{x_0}(\lie{Y})=0$, and hence $S_{x_0}(\lie{Y})\cong
S_{x_0}(\lie{M})\cong S$ by Lemma
\ref{lemma;slice}\eqref{item;symplectization-slice}.  This shows
$S_{x_0}(\lie{X})\cong S$.  Moreover,
$V_{x_0}(\lie{X})\cong\ker(T_{x_0}p)=V$ by Lemma
\ref{lemma;slice}\eqref{item;transverse-slice}.
\end{proof}

\end{alinea}

\begin{alinea}\label{alinea;darboux}
The local normal form theorem is as follows.

\begin{subtheorem}\label{theorem;normal}
Let $x\in X$ and assume that the $G$-action on $X$ is clean at $x$.
Let
$$
\lambda=\Phi(x),\quad H=G_x,\quad S=S_x(X),\quad V=V_x(X),\quad
\a=\n_x.
$$
Choose an $H$-equivariant splitting
$\theta\colon\g_\lambda/\h\to\g_\lambda$ of the quotient map
$\g_\lambda\to\g_\lambda/\h$ which is compatible with $\a$.  Then a
$G$-invariant neighbourhood of $x$ in $X$ is isomorphic as a
presymplectic Hamiltonian $G$-manifold to a $G$-invariant
neighbourhood of $x_0$ in the local model
$\lie{X}=\lie{X}(\lambda,H,\theta,S,V,\a)$.
\end{subtheorem}  

\begin{proof}
First we verify that the list $(\lambda,H,\theta,S,V,\a)$ satisfies
the conditions imposed in
\ref{alinea;model}\eqref{item;covector}--\eqref{item;ideal}.  That the
subgroup $H$ is contained in $G_\lambda$ follows from the equivariance
of the moment map $\Phi$.  That the null ideal $\a$ contains $\lie{k}$
follows from \eqref{equation;null}.  That $\a$ is contained in
$\g_\lambda$ follows from Lemma \ref{lemma;coadjoint}.  That
$\a\cap\h$ acts trivially on the symplectic slice module $S$ is Lemma
\ref{lemma;slice-ideal}.  Now choose a leafwise transitive transversal
$Y$ at $x$ as in Theorem \ref{theorem;transversal} and let $M$ be the
symplectization of $Y$.  Let us denote by $\gamma(X)$ and $\gamma(M)$
the germs of $X$ and $M$ at the orbit $G\cdot x$.  Similarly, let us
denote by $\gamma(\lie{X})$ and $\gamma(\lie{M})$ the germs of
$\lie{X}$ and $\lie{M}=\lie{M}(\lambda,H,S)$ at the orbit $G\cdot
x_0$.  Then
$$S_x(M)\cong S_x(Y)\cong S_x(X)=S\cong S_{x_0}(\lie{M}),$$
where the first two isomorphisms follow from Lemma \ref{lemma;slice}
and the last from Lemma \ref{lemma;model}.  It now follows from the
symplectic local normal form theorem (see
\cite[\S\,41]{guillemin-sternberg;techniques;;1990} or
\cite{marle;modele-action}) that $\gamma(M)$ and $\gamma(\lie{M})$ are
isomorphic as germs of symplectic Hamiltonian $G$-manifolds.
Isomorphisms intertwine moment maps, so from Proposition
\ref{proposition;symplectization} we get that
$\gamma(Y)=\gamma\bigl(\Psi^{-1}(\lambda+\a^\circ)\bigr)$ is
isomorphic to
$$
\gamma\bigl(\Phi_{\lie{M}}^{-1}(\lambda+\a^\circ)\bigr)=
\gamma(\lie{Y}).
$$
Theorem \ref{theorem;transversal} states that $\gamma(X)$ is
isomorphic to the equivariant bundle over $\gamma(Y)$ with fibre $V$,
that is to say the bundle $\gamma\bigl(G\times^H(\q^*\times S\times
V)\bigr)=\gamma(\lie{X})$.
\end{proof}  

\end{alinea}

%%%%%%%%%%%%%%%%%%%%%%%%%%%%%%%%%%%%%%%%%%%%%%%%%%%%%%%%%%%%%%%%%%%%%%%%

\bibliographystyle{amsplain}

\def\cprime{$'$}
\providecommand{\bysame}{\leavevmode\hbox to3em{\hrulefill}\thinspace}
\providecommand{\MR}{\relax\ifhmode\unskip\space\fi MR }
\providecommand{\MRhref}[2]{%
  \href{http://www.ams.org/mathscinet-getitem?mr=#1}{#2}
}
\providecommand{\href}[2]{#2}

%%%%%%%%%%%%%%%%%%%%%%%%%%%%%%%%%%%%%%%%%%%%%%%%%%%%%%%%%%%%%%%%%%%%%%%%

\end{document}